\documentclass{amsart}
\usepackage[utf8]{inputenc}

\usepackage{geometry}
\usepackage{epsdice}

\usepackage{enumitem}

\usepackage{color}

\usepackage[normalem]{ulem}

\usepackage[OT2,T1]{fontenc}

\usepackage[final]{microtype}

\usepackage{tikz}
\usepackage{tikz-cd}
\usepackage{xypic}

\usepackage[draft=false,linktocpage=true]{hyperref}
\usepackage[capitalise]{cleveref}

\usepackage{mathtools}
\usepackage{amsmath}
\usepackage{mathrsfs}

\usepackage{amsthm,amssymb}
\numberwithin{equation}{section}

\theoremstyle{plain}
\newtheorem{theorem}[equation]{Theorem}

\newtheorem{proposition}[equation]{Proposition}
\newtheorem{lemma}[equation]{Lemma}

\newtheorem{corollary}[equation]{Corollary}

\theoremstyle{definition}
\newtheorem{definition}[equation]{Definition}
\newtheorem{question}[equation]{Question}

\newtheorem{example}[equation]{Example}

\newtheorem{remark}[equation]{Remark}

\newtheorem{construction}[equation]{Construction}

\DeclareMathOperator{\Spf}{Spf}
\DeclareMathOperator{\Spa}{Spa}
\DeclareMathOperator{\Fil}{Fil}
\DeclareMathOperator{\Gr}{Gr}

\newcommand{\cC}{\mathcal{C}}
\newcommand{\cF}{\mathcal{F}}
\newcommand{\cO}{\mathcal{O}}

\def\BB{\mathbb{B}}
\def\LL{\mathbb{L}}
\def\NN{\mathbb{N}}
\def\rmD{\mathrm{D}}
\def\rmH{\mathrm{H}}
\def\rmL{\mathrm{L}}
\def\scrC{\mathscr{C}}
\def\scrD{\mathscr{D}}
\def\scrH{\mathscr{H}}
\def\scrU{\mathscr{U}}
\def\scrX{\mathscr{X}}

\def\wh{\widehat}
\def\ra{\rightarrow}
\def\rra{\longrightarrow}

\newcommand{\an}{{\mathrm{an}}}
\newcommand{\ic}{{$\infty$-category }}
\newcommand{\Fun}{{\mathrm{Fun}}}
\newcommand{\colim}{{\mathrm{colim}}}

\newcommand{\CAlg}{{\mathrm{CAlg}}}
\newcommand{\Ch}{{\mathrm{Ch}}}
\newcommand{\dR}{{\mathrm{dR}}}
\newcommand{\pe}{{\mathrm{pro\acute{e}t}}}
\newcommand{\Sh}{{\mathrm{Sh}}}
\newcommand{\Symm}{{\mathrm{Sym}}}
\newcommand{\op}{{\mathrm{op}}}
\newcommand{\st}{{\mathrm{st}}}
\newcommand{\cry}{{\mathrm{crys}}}
\newcommand{\DF}{{\mathrm{DF}}}

\title{Period Sheaves via derived de Rham cohomology}

\author{Haoyang Guo}

\address{Department of Mathematics,University of Michigan, 530 Church Street,
  Ann Arbor, MI 48109}
  
\email{hyguo@umich.edu}

\author{Shizhang Li}

\address{Department of Mathematics,University of Michigan, 530 Church Street,
  Ann Arbor, MI 48109}

\email{shizhang@umich.edu}

\begin{document}

\maketitle

\begin{abstract}
In this article we give an interpretation, in terms of derived
de Rham complexes, of Scholze's de Rham period sheaf
and Tan--Tong's crystalline period sheaf.
\end{abstract}

\tableofcontents

\section{Introduction}
Fontaine's mysterious period rings are essential in formulating various
$p$-adic comparison statements in $p$-adic Hodge theory.
In the past decades there has been an effort to understand these period rings
via other constructions related to differentials.

For instance Colmez realized that one can put a topology on $\overline{\mathbb{Q}_p}$,
related to K\"{a}hler differentials of $\overline{\mathbb{Z}_p}/\mathbb{Z}_p$,
with respect to which the completion becomes the de Rham period ring $B_{\dR}^+$,
see~\cite[Appendix]{Fon94} (which is polished and published in~\cite{Col12}).

Later on Beilinson~\cite[Section 1]{Bei12} 
gives another construction of $B_{\dR}^+$ in terms
of the derived de Rham cohomology (introduced by Illusie in \cite[Chapter VIII]{Ill72})
of $\overline{\mathbb{Q}_p}/\mathbb{Q}_p$.
In terms of our notation, he shows that there is a filtered isomorphism
\[
B_{\dR}^+ \cong \wh\dR^\an_{\overline{\mathbb{Q}_p}/\mathbb{Q}_p};
\]
see~\Cref{ddR construction aff} for the meaning of the right hand side
and~\Cref{Beilinson Colmez example}.\footnote{
For the relation between these two constructions, see~\cite[Proposition 1.6]{Bei12}.}
In a similar vein, Bhatt~\cite[Proposition 9.9]{Bha12} exhibits a filtered isomorphism,
realizing the crystalline period ring
via derived de Rham cohomology of $\overline{\mathbb{Z}_p}/\mathbb{Z}_p$:
\[
A_{\cry} \cong \dR^\an_{\overline{\mathbb{Z}_p}/\mathbb{Z}_p};
\]
see~\Cref{int construction} and~\Cref{Bhatt example}.

Fontaine's period rings admit various generalizations in geometric situations,
for instance see \cite{Fal89}, \cite[Sections 5-6]{Brinon}, \cite[Section 2]{AI13},
\cite[Section 6]{Sch13} and~\cite[Section 2]{TT19}.
From now on let us focus on the ones introduced by Scholze:
recall in his proof of $p$-adic de Rham comparison for smooth proper rigid spaces
over $p$-adic fields~\cite{Sch13}, Scholze introduces period sheaves
$\mathbb{B}_{\dR}^+$ and $\cO\mathbb{B}_{\dR}^+$ (see~\cite[Definition 6.1 and 6.8]{Sch13} and~\cite{SchCor})
on the pro-\'{e}tale site of a smooth rigid space.
However the construction of $\cO\mathbb{B}_{\dR}^+$ is somewhat complicated,
and it takes one a fair amount of effort to understand $\cO\mathbb{B}_{\dR}^+$.
From this understanding Scholze 
deduces a long exact sequence~\cite[Corollary 6.13]{Sch13}:
\[
0 \ra \mathbb{B}_{\dR}^+ \ra \cO\mathbb{B}_{\dR}^+
\xrightarrow{\nabla} \cO\mathbb{B}_{\dR}^+ \otimes_{\cO_X} \Omega^\an_X
\xrightarrow{\nabla} \ldots 
\xrightarrow{\nabla} \cO\mathbb{B}_{\dR}^+ \otimes_{\cO_X} \Omega^{\dim_X, \an}_X \to 0,
\]
known as the $p$-adic analogue of the Poincar\'{e} sequence.
Here $\nabla$ is a connection which behaves like classical Gauss--Manin connection
(satisfying certain Griffiths transversality and so on).

Following the theme, in this article
we explain how to understand Scholze's de Rham period sheaf $\cO\mathbb{B}_{\dR}^+$
in terms of suitable (analytic) derived de Rham sheaves.

Let $k$ be a $p$-adic field. 
In this paper, we introduce the (Hodge-completed)
analytic derived de Rham sheaf $\wh\dR^\an_{X_{\pe}/X}$
for the pro-\'{e}tale site $X_{\pe}$ relative to the analytic site $X$.
Similarly there is also a construction $\wh\dR^\an_{X_{\pe}/k}$ for $X_{\pe}$
relative to $k$. Our main result is the following:
\begin{theorem}[{see~\Cref{BdR+ comparison} and~\Cref{OBdR+ comparison}
for the precise statement}]
Let $X$ be a smooth rigid space over $k$, we have natural filtered isomorphisms:
\[
\mathbb{B}_{\dR}^+  \cong \wh\dR^\an_{X_{\pe}/k} \text{ and }
\cO\mathbb{B}_{\dR}^+ \cong \wh\dR^\an_{X_{\pe}/X}.
\]
\end{theorem}

Moreover in this viewpoint,
one naturally gets the $p$-adic Poincar\'{e} sequence mentioned above.
Indeed, in classical algebraic geometry, suppose 
$X \xrightarrow{f} Y \xrightarrow{g} Z$ is a triangle
of smooth morphisms, then one always has a sequence (see~\cite{KO68}):
\[
0 \to \Omega^*_{X/Z} \to \Omega^*_{X/Y} 
\xrightarrow{\nabla} \Omega^*_{X/Y}  \otimes_{f^{-1}\cO_Y} \Omega^1_{Y/Z}
\xrightarrow{\nabla} \ldots 
\xrightarrow{\nabla} \Omega^*_{X/Y}  \otimes_{f^{-1}\cO_Y} \Omega^{\dim_{Y/Z}}_{Y/Z} \to 0,
\]
whose totalization\footnote{This is only heuristic, as totalizations cannot be made sense at the level of derived category. See~\Cref{Filtration convention} and~\Cref{ddR for a triple}.}, as well as the totalizations of 
the Hodge-graded pieces
(where $\Omega^i_{Y/Z}$ is given degree $i$), are all quasi-isomorphic to $0$.
In the framework of derived de Rham complexes, one has an intuitive base change formula
for a triple of rings $A \to B \to C$:
\[
\dR_{C/A} \otimes_{\dR_{B/A}} B \cong \dR_{C/B},
\]
which leads to a generalization
of the above sequence (see~\Cref{ddR for a triple}).
When one applies this to the triangle $X_{\pe} \to X \to k$, we get
the following re-interpretation of the $p$-adic Poincar\'{e} sequence mentioned above.
\begin{theorem}[{see~\Cref{rational Poincare sequence} for the precise statement}]
Denote $\nu \colon X_{\pe} \to X$ the natural projection from pro-\'{e}tale site
of $X$ to the analytic site of $X$.
The following sequence in $\wh\DF(\wh\dR^\an_{X_{\pe}/k})$:
\[
0 \ra
\wh\dR^\an_{X_{\pe}/k} \ra \wh\dR^\an_{X_{\pe}/X}
\xrightarrow{\nabla} \wh\dR^\an_{X_{\pe}/X} \otimes_{\nu^{-1}\cO_X} \nu^{-1}\Omega^\an_X
\xrightarrow{\nabla} \ldots 
\xrightarrow{\nabla} \wh\dR^\an_{X_{\pe}/X} \otimes_{\nu^{-1}\cO_X} \nu^{-1}\Omega^{\dim_X, \an}_X
\ra 0
\]
is strict exact, where we give $\nu^{-1}\Omega^{i, \an}_X$ degree $i$.
\end{theorem}
Hence in this point of view, the connection $\nabla$ defined by Scholze is indeed
an incarnation of the Gauss--Manin connection.

The advantage of our perspective is that one can naturally generalize the above discussion
to singular rigid spaces. Due to some technical issue, so far we have only worked
out the case where the rigid space $X$ is a local complete intersection over $k$
(see the Appendix~\ref{lci discussion} for a brief discussion of the notion ``l.c.i.'' in rigid geometry).
In this singular case, one no longer gets an ordinary sheaf but rather a sheaf in a
derived $\infty$-category satisfying hyperdescent.
In the local complete intersection case, the hypersheaf $\wh\dR^\an_{X_{\pe}/X}$
is cohomologically bounded below by $-(\text{embedded codimension of }X)$.
However, contemplating with the $0$-dimensional situation in~\Cref{artinian example}, we find that
actually this hypersheaf always lives in cohomological degree $0$ in that situation regardless of the
input Artinian $k$-algebra.
This leads to an interesting question that needs further explorations:
\begin{question}[{same as~\Cref{coh bound question}, c.f.~\cite{Bha12complete}}]
In what generality shall we expect $\wh\dR^\an_{X_\pe/X}$ 
to live in cohomological degree $0$?
And when that happens, can we re-interpret the underlying algebra
via some construction similar to Scholze's $\cO\mathbb{B}_{\dR}^+$
as in~\cite{Sch13} and~\cite{SchCor}?
\end{question}

Finally, we remark that we also have worked out a parallel story related to Tan--Tong's
crystalline period sheaves~\cite[Section 2]{TT19}. 
We summarize the result in this direction as follows.
\begin{theorem}[{see~\Cref{crystalline period comparison} and~\Cref{Crystalline Poincare sequence}
for the precise statements}]
Let $k$ be an absolutely unramified $p$-adic field, with ring of integers $\cO_k$,
and let $\scrX$ be a smooth formal scheme over $\cO_k$.
Denote by $w \colon X_\pe \to \scrX$ the natural projection from the pro-\'{e}tale site
of the rigid generic fiber $X$ of $\scrX$ to the Zariski site of $\scrX$.
Then we have natural filtered isomorphisms:
\[
\mathbb{A}_\cry \cong \dR_{\wh\cO_X^+/\cO_k}^\an \text{ and }
\cO\mathbb{A}_\cry \cong \dR_{\wh\cO_X^+/\cO_\scrX}^\an.
\]
Moreover the following sequence in $\wh\DF(\dR_{\wh\cO_X^+/\cO_k}^\an)$:
\[
0 \ra
\dR_{\wh\cO_X^+/\cO_k}^\an \to \dR_{\wh\cO_X^+/\cO_\scrX}^\an
\xrightarrow{\nabla} \dR_{\wh\cO_X^+/\cO_\scrX}^\an \otimes_{w^{-1}\cO_\scrX }
w^{-1}\Omega^{1,\an}_\scrX \xrightarrow{\nabla} \ldots \xrightarrow{\nabla} 
\dR_{\wh\cO_X^+/\cO_\scrX}^\an \otimes_{w^{-1}\cO_\scrX }
w^{-1}\Omega^{d,\an}_\scrX
\ra 0
\]
is strict exact,
where $d$ is the relative dimension of $\scrX/\cO_k$
and $w^{-1}\Omega^{i,\an}_\scrX$ is given degree $i$.
\end{theorem}

We want to mention that in our situation, 
we mostly care about the analytic derived de Rham complex for a map of adic spaces $X\ra Y$, 
where $X$ is a perfectoid space and $Y$ is a rigid space (or their integral analogues).
The analytic derived de Rham complex for a map of rigid spaces have been studied independently in~\cite{Ant20}
and a forthcoming article~\cite{Guo} by the first named author.

Let us give a brief summary of the content of the following sections.
In~\Cref{Notation and Conventions} we explain notation and conventions used in this paper,
and we give a brief discussion of relevant facts about filtered derived $\infty$-categories
and sheaves in them.
In~\Cref{Integral theory} and~\Cref{Rational Theory} we work out,
in a parallel way, the realizations of Scholze's and Tan--Tong's period sheaves.
In both sections, we first introduce the relevant algebraic construction,
then discuss the Poincar\'{e} sequence, and finally globalize (or sheafify) these constructions
and show that they are (essentially) the same as aforementioned period sheaves.
In Appendix~\ref{lci discussion} we make a primitive discussion
of local complete intersections in rigid geometry.

\subsection*{Acknowledgement}
We are grateful to Bhargav Bhatt for suggesting this project to us
as well as many discussions related to it.
We thank David Hansen heartily for listening to our project,
drawing our attention to Tan--Tong's period sheaves,
and sharing excitement with the second named author.
The second named author would also like to thank Johan de Jong
for discussions surrounding the notion of local complete intersections
in rigid geometry and suggesting the proof of a Lemma.
The first named author is partially funded by Department of Mathematics, University of Michigan, 
and by NSF grant DMS 1801689 through Bhargav Bhatt.

\section{Notation and Conventions}
\label{Notation and Conventions}
\subsection{Notation}
We fix $k$ to be a complete discretely valued $p$-adic field with a perfect residue field, 
and let $\cO_k$ be its ring of integers.
Denote by $\Spa(k)$ to be the adic spectrum $\Spa(k,\cO_k)$. 

Anything with the superscript decoration $(-)^\an$ will mean a suitably $p$-completed version
of the classical object $(-)$.
The sense in which we are taking $p$-completion of these objects shall be clear from the context.

The tensor products $\otimes$ appear in this article, if not otherwise specified, 
always denote derived tensor products.
Similarly, the completed tensor products appear always indicate derived completion
of the derived tensor product (with respect to suitable filtrations
to be specified in each case).

\subsection{Filtrations}
\label{Filtration convention}
Many objects we are dealing with in this article are viewed as objects either in
the filtered derived $\infty$-category $\DF(R) \coloneqq \Fun(\mathbb{N}^\op,\scrD(R))$ 
or in the full derived $\infty$-subcategory $\wh\DF(R) \subset \DF(R)$ 
consisting of objects that are derived complete with respect to the filtration,
for some ring $R$ which should be clear from the context.
For a brief introduction of these, we refer readers to~\cite[Subsection 5.1]{BMS2}.

We need a notion of~\emph{step sequence functor}, which is perhaps a non-standard terminology.
Given an integer $i \in \mathbb{N}$, we have a functor $\Gr^i \colon \DF(R) \ra \scrD(R)$
sending a filtered object to its $i$-th graded piece.
This functor has a right adjoint which we call~\emph{the $i$-th step sequence functor} 
and denote it by $\st_i \colon \scrD(R) \ra \DF(R)$.
Concretely, the value of $\st_i(C)$ on $j$ is given by
\[
C_j = 
\begin{cases}
C;~0 \leq j \leq i;\\
0;~else.
\end{cases}.
\]

Let $\cC$ be a stable $\infty$-category,
for example $\cC$ could be $\scrD(R)$, $\DF(R)$ or $\wh\DF(R)$ for a discrete ring $R$.
	Consider a sequence of objects in $\cC$
	\[
	A_0 \xrightarrow{d_0} A_1 \xrightarrow{d_1} A_2 \xrightarrow{d_2} \ldots 
	\]
	such that $d_{i+1}\circ d_i=0$. 
	If there exists an object $L$ in the filtered $\infty$-category $\Fun(\NN^\op, \cC)$, 
	satisfying the following conditions
	\begin{itemize}
		\item $L(0)=A_0$;
		\item $L(i)/L(i+1) \cong A_{i+1}[-i]$;
		\item the natural map $L(0) \to L(0)/L(1)$ is identified with $d_0$;
		\item the natural connecting map of graded pieces $L(i)/L(i+1)\ra L(i+1)/L(i+2)[1]$ is isomorphic to $d_{i+1}[-i]$,
	\end{itemize}
then we say the sequence is \emph{witnessed by the filtration $L$ on $A_0$}.
The notion is an $\infty$-analogue of a complex in the chain complex category.

When $\cC=\DF(R)$, then $L$ can be regarded as an object 
$G(\bullet, \bullet) \in \Fun((\NN \times \NN)^\op, \scrD(R))$, 
where we use the convention that we denote the first coordinate by $i$, the second coordinate
by $j$, and $L(i) = G(i,0)$.
In this setting, we say the filtration $L(\bullet)$ on $A_0$ 
is \emph{strict exact} if for any $j \in \NN$,
the object $G(0,j)$ is complete with respect to the filtration $G(i,j)$.
Assume all of the $A_{i} = G(i-1,0)/G(i,0)[i-1]$ are cohomologically
supported in degree $0$ with filtrations (coming from the second coordinate)
given by actual $R$-submodules.
Then the sequence of $A_i$'s above
can be thought of as a sequence of ordinary filtered $R$-modules,
and our notion of strict exactness defined here agrees with the classical notion 
of strict exactness of a sequence of filtered $R$-modules.


\subsection{Sheaves and hypersheaves}
Here we give a quick review about sheaves in $\infty$-category.

Let $X$ be a site, and let $\scrC$ be a presentable $\infty$-category.
The \ic of presheaves in $\scrC$, denoted as $\mathrm{PSh}(X,\scrC)$, 
is defined to be the \ic $\Fun(X^\op,\scrC)$ of contravariant functors from $X$ to $\scrC$.
The \ic $\mathrm{PSh}(X,\scrC)$ admits a full sub \ic $\Sh(X,\scrC)$ of \emph{(infinity) sheaves in $\scrC$},
consisting of functors $\cF:X^\op \ra \scrC$ 
that send (finite) coproducts to products and satisfy the descent along \v{C}ech nerves: 
for any covering $U'\ra U$ in $X$, 
the natural morphism to the limit below is required to be a weak equivalence
\[
\cF(U) \rra \lim_{[n]\in \Delta^\op} \cF(U'_n), \tag{$\ast$}
\]
where $U'_\bullet \ra U$ is the \v{C}ech nerve associated with the covering $U'\ra U$.
Here we note that this is the $\infty$-categorical analogue of the classical sheaf condition in ordinary categories.
	
There is a stronger descent condition which requires $(\ast)$ above to hold
with respect to all \emph{hypercovers} $U'_\bullet \ra U$ in the site $X$.
Sheaves satisfying such stronger condition are called \emph{hypersheaves}.
For example, given any bounded below complex $C$ of ordinary sheaves on a site $X$,
the assignment $U \mapsto \mathrm{R}\Gamma(U, C)$
gives rise to a hypersheaf.
The collection of hypersheaves in $\scrC$ forms a full 
sub-\ic $\Sh^{{\mathrm{hyp}}}(X,\scrC)$ inside $\Sh(X,\scrC)$.

\begin{remark}\label{Rmk, equiv}
Let $\scrC=\scrD(R)$ be the derived $\infty$-category of $R$-modules.
Then the $\infty$-category $\Sh^{\mathrm{hyp}}(X,\scrC)$ of hypersheaves over $X$ is in fact equivalent to the derived $\infty$-category $\scrD(X,R)$ of classical sheaves of $R$-modules over $X$, by \cite[Corollary 2.1.2.3]{Lu18}.
Here the functor $\scrD(X,R) \ra \Sh^{\mathrm{hyp}}(X,\scrC)$ 
associates a complex of ordinary sheaves $C$ with the functor
\[
U \mapsto R\Gamma(U, C),~\forall~U\in X.
\]
As an upshot, the underlying homotopy category of $\Sh^{\mathrm{hyp}}(X,\scrC)$ is the classical derived category of sheaves of $R$-modules over $X$.
In particular, given a hypersheaf $\cF$ of $R$-modules over $X$, we can always represent it by an actual complex of sheaves of $R$-modules.
\end{remark}

\subsection{Unfold a hypersheaf}
\label{unfolding}
There is a way to define a hypersheaf on a site $X$ via unfolding from a basis, 
c.f.~\cite[Proposition 4.31]{BMS2} and the discussion after it.

Let $X$ be a site and let $\mathcal{B}$ be a basis of $X$,
namely $\mathcal{B}$ is a subcategory of $X$ such that for each object $U$ in $X$, 
there exists an object $U'$ in $\mathcal{B}$ covering $U$.
So any hypercover of an object in $X$ can be refined to a hypercover with each term in $\mathcal{B}$.
Let $\scrC$ be a presentable $\infty$-category.

Let $\cF\in \Sh^{\mathrm{hyp}}(\mathcal{B},\scrC)$ be a hypersheaf on $\mathcal{B}$.
We can then \emph{unfold} the sheaf $\cF$ to a hypersheaf $\cF'$ on $X$, 
such that its evaluation at any $V\in X$ is given by
\[
\cF'(V)=\underset{U'_\bullet\ra V}{\colim}\varprojlim_{[n]\in \Delta^\op} \cF(U'_n),
\]
where the colimit is indexed over all hypercovers $U'_\bullet\ra V$ with $U'_n\in \mathcal{B}$ for all $n$.
It can be shown that one hypercover suffices to compute the value of $\cF'(V)$ in the above formula:
actually for a hypercover $U'_\bullet\ra V$ with each $U'_n$ in the basis $\mathcal{B}$, 
we have a natural weak-equivalence
\[
\lim_{[n]\in \Delta^\op} \cF(U'_n) \rra \cF'(V).
\]
In particular for any $U \in \mathcal{B}$, the natural map  $\cF(U) \rra \cF'(U)$ is a weak-equivalence. 

The above construction is functorial with respect to $\cF\in \Sh^{\mathrm{hyp}}(\mathcal{B}, \scrC)$, 
and we get a natural unfolding functor
\[
\Sh^{\mathrm{hyp}}(\mathcal{B},\scrC) \rra \Sh^{\mathrm{hyp}}(X, \scrC),
\]
which is in fact an equivalence, with the inverse given by the restriction functor 
$\Sh^{\mathrm{hyp}}(X,\scrC)\ra \Sh^{\mathrm{hyp}}(\mathcal{B},\scrC)$.

\section{Integral theory}
\label{Integral theory}

\subsection{Affine construction}

In this subsection we define analytic cotangent complex and analytic 
derived de Rham complex for a morphism of $p$-adic algebras.
We refer readers to~\cite[Sections 2 and 3]{Bha12} for general background of the derived de Rham
complex in a $p$-adic situation.

\begin{construction}[Integral constructions]
\label{int construction}
Let $A_0\ra B_0$ be a map of $p$-adically complete algebras over $\cO_k$, 
and $P$ be the standard polynomial resolution of $B_0$ over $A_0$.
	
We define the \emph{analytic cotangent complex of $A_0\ra B_0$,} 
denoted as  $\LL_{B_0/A_0}^\an$, to be the derived $p$-completion of the complex $\Omega_{P/A_0}^1\otimes_{P} B_0$ of $B_0$-modules.

Next we denote $(|\Omega_{P/A_0}^*|,\Fil^*)$ 
the direct sum totalization of the simplicial complex 
$\Omega_{P/A_0}^*$ together with its Hodge filtration, as an object in $\Fun(\NN^\op,\Ch(A_0))$.
As the de Rham complex of a simplicial ring admits a commutative differential graded algebra structure, 
we may regard $|\Omega_{P/A_0}^*|$ with its Hodge filtration as an object in $\CAlg(\Fun(\NN^\op, \Ch(A_0)))$.
Then the \emph{analytic derived de Rham complex of $B_0/A_0$}, 
denoted as $\dR_{B_0/A_0}^\an$ in the $\CAlg(\DF(A_0))$, is defined as the derived $p$-completion of the filtered cdga 
$(|\Omega_{P/A_0}^*|,\Fil^*)$.

\end{construction}

\begin{remark}
By construction, the graded pieces of the derived Hodge filtrations of 
$\dR_{B_0/A_0}^\an$ are given by
$$\Gr^i(\dR_{B/A}^\an) \cong (\rmL\wedge^i\LL_{B/A})^\an[-i],$$
where $\rmL \wedge^i$ denotes the $i$-th left derived wedge product, 
c.f.~\cite[Construction 4.1]{Bha12complete}.
\end{remark}

Let us establish some properties of this construction before discussing any example.

\begin{lemma}
\label{general commutative diagram for a triangle}
Let $A \to B \to C$ be a triple of rings, then we have a commutative diagram
of filtered $E_\infty$ algebras:
\[
\xymatrix{
\dR_{B/A} \ar[d] \ar[r] & \dR_{C/A} \ar[d] \\
B \ar[r] & \dR_{C/B},
}
\]
where the left arrow is the projection to $0$-th graded piece of the derived Hodge filtration,
and the other three arrows come from functoriality of the construction of derived de Rham complex.
\end{lemma}

\begin{proof}
This follows from left Kan extension of
the case when $B$ is a polynomial $A$-algebra and $C$ is a polynomial
$B$-algebra.
\end{proof}

The following is the key ingredient in understanding the analytic derived de Rham complex
in situations that are interesting to us.

\begin{theorem}
\label{dR when relatively perfect}
Let $A \to B \to C$ be ring homomorphisms of $p$-completely flat $\mathbb{Z}_p$-algebras,
such that $A/p \to B/p$ is relatively perfect (see~\cite[Definition 3.6]{Bha12}).
Then we have
\begin{enumerate}
\item $\LL_{B/A}^\an = 0$, and $\dR_{B/A}^\an = B$;
\item The natural map $\dR_{C/A}^\an \to \dR_{C/B}^\an$ is an isomorphism;
\item We have a commutative diagram:
\[
\xymatrix{
\dR_{B/A}^\an \ar[d]^{\cong} \ar[r] & \dR_{C/A}^\an \ar[d]^{\cong} \\
B \ar[r] & \dR_{C/B}^\an.
}
\]
\item Assume furthermore that $B \to C$ is surjective with kernel $I$
and $B/p \to C/p$ is a local complete intersection, 
then the natural map $B \to \dR_{C/B}^\an$ exhibits the latter as
$D_{B}(I)^\an$,
the $p$-adic completion of the PD envelope of $B$ along $I$.
Moreover the $p$-adic completion of the PD filtrations $\Fil^r = I^{[r], \an}$
are identified with the $r$-th Hodge filtration.
\end{enumerate}
\end{theorem}

Note that by~\cite[Lemma 3.38]{Bha12} $D_B(I)^\an$ is a $p$-complete
flat $\mathbb{Z}_p$-algebra. Hence 
$I^{[r], \an}$, being submodules of a flat $\mathbb{Z}_p$-module,
are also $p$-torsionfree for all $r$.

\begin{proof}
(1) and (2) follow from the proof of~\cite[Corollary 3.8]{Bha12}:
one immediately reduces modulo $p$ and appeals to the conjugate filtration.
(3) follows from~\Cref{general commutative diagram for a triangle} by taking
the derived $p$-completion.

As for (4), we first apply~\cite[Proposition 3.25]{Bha12}
and~\cite[Th\'{e}or\`{e}me V.2.3.2]{Ber74} to see that
there is a natural filtered map 
$\mathscr{C}omp_{C/B} \colon \dR_{C/B}^\an \to D_B(I)^\an$
such that precomposing with $B \to \dR_{C/B}^\an$ gives the natural map
$B = B^\an \to D_B(I)^\an$.
By~\cite[Theorem 3.27]{Bha12} we see that $\mathscr{C}omp_{C/B}$ is
an isomorphism for the underlying algebra.
To show the same holds for filtrations, it suffices to show 
that the induced map on graded pieces are isomorphisms as the map is compatible with
filtrations. To that end, by a standard spread out technique, 
we may reduce to the case
where $B$ is the $p$-adic completion of a finite type $\mathbb{Z}_p$ algebra,
in particular it is Noetherian, in which case the identification
of graded pieces via this natural map follows from a result of
Illusie~\cite[Corollaire VIII.2.2.8]{Ill72}.
\end{proof}

Now we are ready to do some examples.
An inspiring arithmetic example is worked out by Bhatt.
\begin{example}[{\cite[Proposition 9.9]{Bha12}}]
\label{Bhatt example}
There is a filtered isomorphism:
\[
A_{\mathrm{crys}} \cong \dR^\an_{\overline{\mathbb{Z}_p}/\mathbb{Z}_p}.
\]
\end{example}

Let us work out a geometric example below.

\begin{example}
\label{cotangent complex of perfectoid circle}
Let $n$ be a positive integer.
Let $R = \mathbb{Z}_p \langle T_1^{\pm 1}, \ldots, T_n^{\pm 1} \rangle$, and 
$R_{\infty} = \mathbb{Z}_p \langle T_1^{\pm 1/p^{\infty}}, 
\ldots, T_n^{\pm 1/p^{\infty}} \rangle = 
R \langle S_1^{1/p^{\infty}}, \ldots, S_n^{1/p^{\infty}} \rangle/(T_i - S_i; 
1 \leq i \leq n)$.

Applying (derived $p$-completion of) the fundamental triangle of cotangent complexes to
$$
\mathbb{Z}_p \to R \to R_{\infty},
$$
one yields that $\LL_{R_{\infty}/R}^{\an} = R_{\infty} \cdot \{dT_1, \ldots, dT_n\} [1]$.

On the other hand, the fundamental triangle associated with
$$
R \to R \langle S_1^{1/p^{\infty}}, \ldots, S_n^{1/p^{\infty}} \rangle 
\to R_{\infty}
$$
gives us $\LL_{R_{\infty}/R}^{\an} = R_{\infty} \cdot \{T_i - S_i; 1 \leq i \leq n\} [1]$.

The relation between these two presentations of $\LL_{R_{\infty}/R}^{\an}$ is that
\[
T_i - S_i = dT_i
\]
in $\rmH_{1}(\LL_{R_{\infty}/R}^{\an})$,
as $\frac{\partial}{\partial T_i} (T_i - S_i) = 1$.\footnote{
Here we follow the sign conventions in the Stacks Project,
see~\cite[\href{https://stacks.math.columbia.edu/tag/07MC}{Tag 07MC} footnote 1]{Sta}}
\end{example}

Following the above notation,
we describe $\dR^\an_{R_{\infty}/R}$.

\begin{example}
\label{dR of perfectoid circle}
Applying~\Cref{dR when relatively perfect} to $A = R, 
B = R \langle S_1^{1/p^{\infty}}, \ldots, S_n^{1/p^{\infty}} \rangle$ 
and $I = (T_1 - S_1, \ldots, T_n - S_n)$, we see that
$\dR^\an_{R_{\infty}/R} = 
\left( \rmD_{\mathbb{Z}_p \langle T_1^{\pm 1}, \ldots, T_n^{\pm 1}, S_1^{1/p^\infty},
\ldots, S_n^{1/p^\infty} \rangle} (I) \right)^{\an}$ 
is the $p$-adic completion of the PD envelope of
$R \langle S_1^{1/p^\infty}, \ldots, S_n^{1/p^\infty} \rangle$ 
along $I$ (notice that the PD envelope
is $p$-torsion free, hence derived completion agrees with classical completion),
and the Hodge filtrations are ($p$-adically) 
generated by divided powers of $\{T_i - S_i\}$.
\Cref{cotangent complex of perfectoid circle} shows that the image of $(T_i - S_i)$
in $\Gr^1 = \LL_{R_{\infty}/R}^{\an}[-1] = 
R_{\infty} \otimes_{R} \Omega^{1,\an}_{R/\mathbb{Z}_p}$ is identified with 
$1 \otimes dT_i$.
This precise identification will be used later (see~\Cref{rational torus example} 
and the proof of~\Cref{OBdR+ comparison})
when we compare certain rational
version of the analytic derived de Rham complex with Scholze's period sheaf $\cO\BB_{\dR}^{+}$.
\end{example}

\subsection{derived de Rham complex for a triple}
\label{ddR for a triple}
Given a pair of smooth morphisms
$A \to B \to C$, there is a natural Gauss--Manin connection
$\dR_{C/B} \xrightarrow{\nabla} \dR_{C/B} \otimes_B \Omega^1_{B/A}$,
such that $\dR_{C/A}$ is naturally identified with the ``totalization'' of the following
sequence:
\[
\dR_{C/B} \xrightarrow{\nabla} \dR_{C/B} \otimes_B \Omega^1_{B/A} \xrightarrow{\nabla}
\cdots \xrightarrow{\nabla} \dR_{C/B} \otimes_B \Omega^{\dim_{B/A}}_{B/A}.
\]
Katz and Oda~\cite{KO68} observed that this can be explained by a filtration on $\dR_{C/A}$.
In this subsection we shall show how to generalize this to the context
of derived de Rham complex for a pair of arbitrary morphisms $A \to B \to C$.

We first need to introduce a way to attach filtration on a tensor product
of filtered modules over a filtered $E_{\infty}$-algebra.
The following fact about Bar resolution is well-known, 
and we thank Bhargav Bhatt for teaching us in this generality.
\begin{lemma}
\label{Bar resolution}
Let $A$ be an ordinary ring, let $R$ be an $E_{\infty}$-algebra over $A$,
and let $M$ and $N$ be two objects in $\scrD(R)$.
Then the following augmented simplicial object in $\scrD(A)$
\[
\left(
\xymatrix{
\cdots \ar@<1.2ex>[r] \ar@<.4ex>[r] \ar@<-.4ex>[r] \ar@<-1.2ex>[r] &
M \otimes_A R \otimes_A R \otimes_A N \ar@<0.8ex>[r] \ar[r] \ar@<-.8ex>[r] &
M \otimes_A R \otimes_A N \ar@<.4ex>[r] \ar@<-.4ex>[r] &
M \otimes_A N
}
\right) \longrightarrow M \otimes_R N
\]
displays $M \otimes_R N$ as the colimit of the simplicial objects in $\scrD(A)$.
Here the arrows are given by ``multiplying two factors together''.
\end{lemma}

\begin{proof}
Since the $\infty$-category $\scrD(R)$ is generated by shifts of $R$~\cite[7.1.2.1]{Lu17},
commuting tensor with
colimit, we may assume that both of $M$ and $N$ are just $R$.
In this case, the statement holds for merely $E_1$-algebras,
as we have a null homotopy
$R^{\otimes_A n} \to R^{\otimes_A (n+1)}$ given by tensoring $R^{\otimes_A n}$
with the natural map $A \to R$.
\end{proof}

\begin{construction}
\label{mixed filtration}
Let $A$ be an ordinary ring, let $R$ be a filtered $E_\infty$ algebra over $A$,
and let $M$ and $N$ be two filtered $R$-modules with filtrations compatible
with that on $R$.
Then we regard $M \otimes_R N$ as an object in $\DF(A)$ via
the Bar resolution in~\Cref{Bar resolution}, with
\[
\Fil^i(M \otimes_R N) \coloneqq \colim_{\Delta^{\op}}
\left(
\xymatrix{
\cdots \ar@<1.2ex>[r] \ar@<.4ex>[r] \ar@<-.4ex>[r] \ar@<-1.2ex>[r] &
\Fil^i(M \otimes_A R \otimes_A R \otimes_A N) \ar@<0.8ex>[r] \ar[r] \ar@<-.8ex>[r] &
\Fil^i(M \otimes_A R \otimes_A N) \ar@<.4ex>[r] \ar@<-.4ex>[r] &
\Fil^i(M \otimes_A N)
}
\right),
\]
where the filtrations on $M \otimes_A R \otimes_A \cdots \otimes_A R \otimes_A N$ are given
by the usual Day involution.
\end{construction}

\begin{lemma}
Let $A, R, M, N$ be as in~\Cref{mixed filtration}.
Then we have
\[
\Gr^*(M \otimes_R N) \cong \Gr^*(M) \otimes_{\Gr^*(R)} \Gr^*(N).
\]
\end{lemma}

\begin{proof}
We have
\[
\Gr^*(M \otimes_R N) \cong \colim_{\Delta^{\op}}
\left(
\xymatrix{
\cdots \ar@<1.2ex>[r] \ar@<.4ex>[r] \ar@<-.4ex>[r] \ar@<-1.2ex>[r] &
\Gr^*(M \otimes_A R \otimes_A R \otimes_A N) \ar@<0.8ex>[r] \ar[r] \ar@<-.8ex>[r] &
\Gr^*(M \otimes_A R \otimes_A N) \ar@<.4ex>[r] \ar@<-.4ex>[r] &
\Gr^*(M \otimes_A N)
}
\right)
\]
\[\cong \colim_{\Delta^{\op}}
\left(
\xymatrix{
\cdots 
\ar@<0.8ex>[r] \ar[r] \ar@<-.8ex>[r] &
\Gr^*(M) \otimes_A \Gr^*(R) \otimes_A \Gr^*(N) \ar@<.4ex>[r] \ar@<-.4ex>[r] &
\Gr^*(M) \otimes_A \Gr^*(N)
}
\right)
\cong \Gr^*(M) \otimes_{\Gr^*(R)} \Gr^*(N).
\]
\end{proof}

\begin{proposition}
\label{ddR base change formula}
Let $A \to B \to C$ be a triple of rings, then the diagram of filtered $E_\infty$-algebras
in~\Cref{general commutative diagram for a triangle}
induces a filtered isomorphism of filtered $E_{\infty}$-algebras over $B$:
\[
\dR_{C/A} \otimes_{\dR_{B/A}} B \cong \dR_{C/B}.
\]
\end{proposition}

Here the left hand side is equipped with the filtration in~\Cref{mixed filtration}
with the Hodge filtrations on $\dR_{C/A}$ and $\dR_{B/A}$, and $\Fil^i(B) = 0$
for $i \geq 1$.
The right hand side is equipped with the Hodge filtration.
Denote $\Omega^*_{B/A} \coloneqq \oplus_i \st_i(\rmL\wedge^i\LL_{B/A})[-i]$
the graded algebra associated with the Hodge filtration.

\begin{proof}
After cofibrant replacing $B$ by a simplicial polynomial $A$-algebra
and $C$ by a simplicial polynomial $B$-algebra, we reduce the statement
to the case where $B$ is a polynomial $A$-algebra and $C$ is a polynomial $B$-algebra.
One verifies directly that in this case we have
\[
\dR_{C/A} \otimes_{\dR_{B/A}} B \cong \dR_{C/B} \text{ and }
\Omega^*_{C/A} \otimes_{\Omega^*_{B/A}} B \cong \Omega^*_{C/B}.
\]
Now we finish proof by recalling that a filtered morphism with isomorphic
underlying object is a filtered isomorphism if and only if the induced
morphisms of graded pieces are isomorphisms.
\end{proof}

\begin{construction}
\label{Katz--Oda filtration}
Let $A \to B \to C$ be a triple of rings, then we put a filtration on $\dR_{C/A}$
by the following:
$L(i) = \dR_{C/A} \otimes_{\dR_{B/A}} \Fil^i_{\mathrm{H}}(\dR_{B/A})$,
viewed as a commutative algebra object in $\Fun(\NN^\op, \DF(A)) = 
\Fun((\NN \times \NN)^\op, \scrD(A))$,
where the filtration on $L(i)$ is as in~\Cref{mixed filtration} with each factor
being equipped with its own Hodge filtrations.
We have $L(0) \cong \dR_{C/A}$, and we call $L(i)$ \emph{the $i$-th Katz--Oda filtration
on $\dR_{C/A}$}, and we shall denote it by $\Fil^i_{\mathrm{KO}}(\dR_{C/A})$.
\end{construction}

We caution readers that each $\Fil^i_{\mathrm{KO}}(\dR_{C/A})$ is equipped
with yet another filtration, we shall still call it the Hodge filtration,
the index is often denoted by $j$.
The graded pieces of the Katz--Oda filtration when both arrows in $A \to B \to C$ are smooth
were studied by Katz--Oda~\cite{KO68}, although in a different language, hence the name.
\begin{lemma}
\label{properties of KO filtration}
Let $A \to B \to C$ be a triple of rings, then
\begin{enumerate}
\item We have a filtered isomorphism
\[
\Gr^i_{\mathrm{KO}}(\dR_{C/A}) \cong \dR_{C/B} \otimes_B
\st_i((\rmL\wedge^i\LL_{B/A})[-i]).
\]
\item Under the above filtered isomorphism, the Katz--Oda filtration on $\dR_{C/A}$
witnesses the following sequence:
\[
\dR_{C/A} \to \dR_{C/B} \xrightarrow{\nabla} \dR_{C/B} \otimes_{B} \st_1(\LL_{B/A})
\xrightarrow{\nabla} \cdots
\]
Here $\nabla$ denotes connecting homomorphisms,
which is $\dR_{C/A}$-linear and satisfies Newton--Leibniz rule.
\item The induced Katz--Oda filtration on $\Gr^j_{\mathrm{H}}(\dR_{C/A})$ is complete.
In fact $\Fil^i_{\mathrm{KO}}\Gr^j_{\mathrm{H}}(\dR_{C/A}) = 0$
whenever $i > j$.
\item If $A \to B$ is smooth of equidimension $d$, then 
$\Fil^i_{\mathrm{KO}}\Fil^j_{\mathrm{H}}(\dR_{C/A}) \cong 0$ for any $i> d$.
In particular, combining with the previous point, we get that in this situation 
the Katz--Oda filtration is strict exact in the sense of~\Cref{Filtration convention}.
\end{enumerate} 
\end{lemma}

\begin{proof}
For (1): we have
\[
\Gr^i_{\mathrm{KO}}(\dR_{C/A}) \cong \dR_{C/A} \otimes_{\dR_{B/A}}
\st_i(\rmL\wedge^i\LL_{B/A})[-i] \cong (\dR_{C/A} \otimes_{\dR_{B/A}} B) \otimes_B
\st_i(\rmL\wedge^i\LL_{B/A})[-i],
\]
and by~\Cref{ddR base change formula} the right hand side can be identified with
$\dR_{C/B} \otimes_B \st_i(\rmL\wedge^i\LL_{B/A})[-i]$.

For (2): we just need to show the properties of these $\nabla$'s.
With any multiplicative filtration on an $E_\infty$-algebra $R$,
we get a natural filtered map $\Fil^i \otimes_R \Fil^j \to \Fil^{i+j}(R)$
where the left hand side is equipped with the Day convolution filtration (over the
underlying algebra $R$).
Now we look at the following commutative diagram:
\[
\xymatrix{
(\Gr^{i} \otimes_R \Gr^{j+1}) \oplus (\Gr^{i+1} \otimes_R \Gr^j) \ar[r] \ar[d] &
\Fil^{i+j}/\Fil^{i+j+2}(\Fil^i \otimes_R \Fil^j) \ar[r] \ar[d] &
\Gr^i \otimes_R \Gr^j \xrightarrow{+1} \ar[d] \\
\Gr^{i+j+1} \ar[r] &
\Fil^{i+j}/\Fil^{i+j+2}(R) \ar[r] &
\Gr^{i+j} \xrightarrow{+1}
}
\]
to conclude that the connecting morphisms are $R$-linear and satisfy
Newton--Leibniz rule. 
Since $\Fil^i_{\mathrm{KO}}$ is a multiplicative filtration on $\dR_{C/A}$,
we get the desired properties of $\nabla$.

(3) follows from the distinguished triangle of cotangent complexes and their
exterior powers.

(4) follows from the definition of the Katz--Oda filtration in~\Cref{Katz--Oda filtration} 
and the fact that $\Fil^i_{\mathrm{H}}(\dR_{B/A}) = 0$ whenever $i > d$.
\end{proof}

We do not need the following construction in this paper,
but mention it for the sake of completeness of our discussion.
\begin{construction}
We denote the graded algebra associated with the Hodge filtration on derived de Rham
complex by $\rmL\Omega^*_{-/-}$.\footnote{We warn readers that this is not a standard
notation, in other literature the symbol $\rmL\Omega$ is often used to denote
the derived de Rham complex.}
Let $A \to B \to C$ be a triple of rings. 
Note that 
$\rmL\Omega^*_{C/A} \cong \rmL\wedge^*_{C}(\st_1(\LL_{C/A}))[-*]$,
and we have a functorial filtration
$\LL_{B/A} \otimes_B C \to \LL_{C/A}$ with quotient being $\LL_{C/B}$.
Hence there is a functorial multiplicative exhaustive increasing filtration
on $\rmL\Omega^*_{C/A}$,
called the \emph{vertical filtration} and denoted by $\Fil^v_i$, 
consisting of graded-$\rmL\Omega^*_{B/A}$-submodules
with graded pieces given by
$\Gr^v_i = \rmL\Omega^*_{B/A} \otimes_B \st_i(\rmL\wedge^{i}\LL_{C/B})[-i]$.
\end{construction}

Let us summarize the picture of (the graded pieces of) these filtrations in the following
diagram:
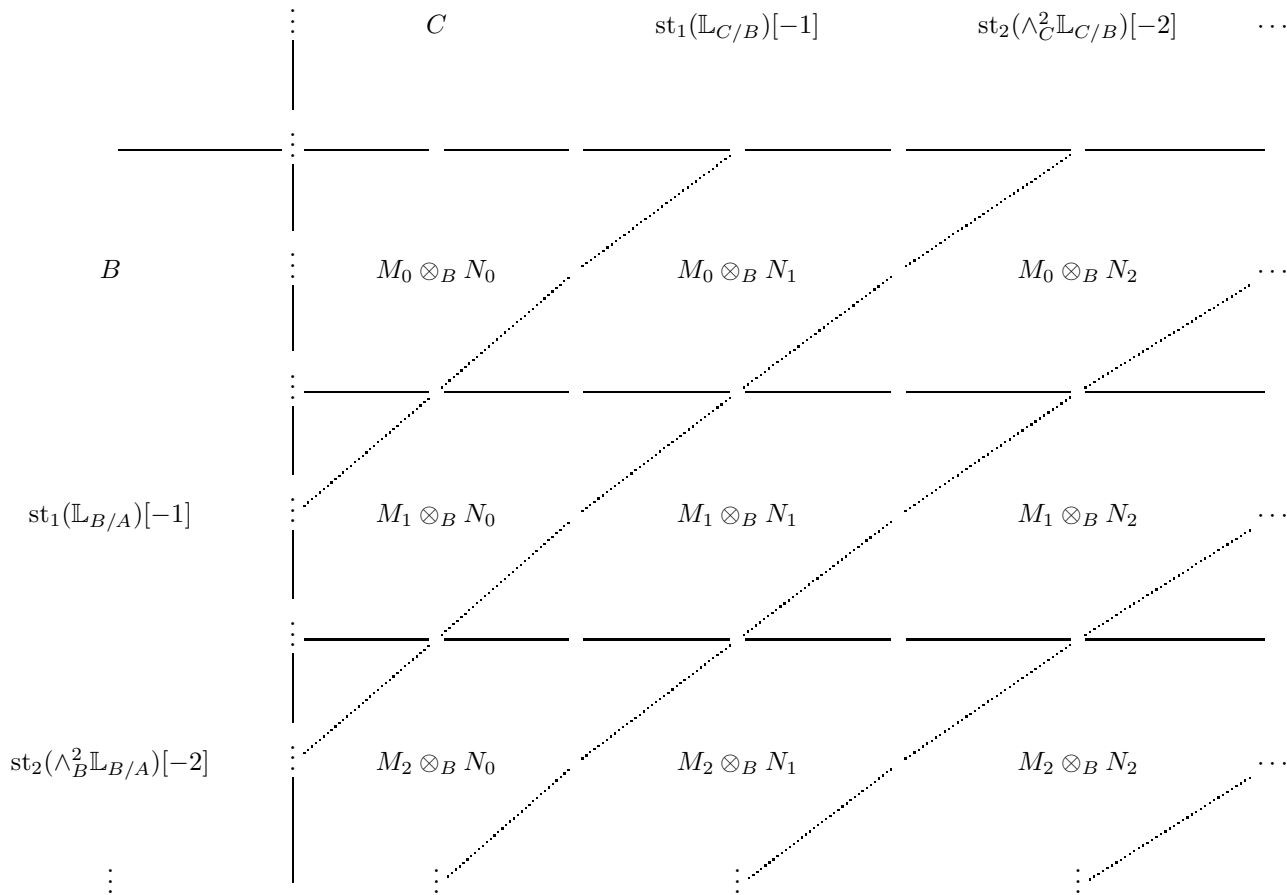
\[
\xymatrix{
  & \vdots \ar@{-}[d] & C  &  & \st_1(\LL_{C/B})[-1]  &   & 
  \st_2(\wedge^2_C \LL_{C/B})[-2]  &  \cdots  \\
\ar@{-}[r] & \vdots \ar@{-}[d] \ar@{-}[l] \ar@{-}[r] & \ar@{-}[l]  \ar@{-}[r]  & \ar@{-}[r] & \ar@{-}[r] \ar@{.}[ld]  & \ar@{-}[r] & \ar@{-}[r] \ar@{.}[ld]   &    \\
B  & \vdots \ar@{-}[d] &  M_0 \otimes_B N_0  & \ar@{.}[ld] &  M_0 \otimes_B N_1  & \ar@{.}[ld]  &  M_0 \otimes_B N_2 & \ar@{.}[ld]  \cdots  \\
 & \vdots \ar@{-}[d] & \ar@{-}[l] \ar@{.}[ld]  \ar@{-}[r]  & \ar@{-}[r] & \ar@{-}[r] \ar@{.}[ld]  & \ar@{-}[r] & \ar@{-}[r] \ar@{.}[ld]   &    \\
\st_1(\LL_{B/A})[-1] & \vdots \ar@{-}[d]  &  M_1 \otimes_B N_0   & \ar@{.}[ld] &  
M_1 \otimes_B N_1   & \ar@{.}[ld] &  M_1 \otimes_B N_2 & \ar@{.}[ld]  \cdots  \\
 & \vdots \ar@{-}[d] &  \ar@{-}[l] \ar@{-}[r] \ar@{.}[ld]  &  \ar@{-}[r] &  \ar@{-}[r] \ar@{.}[ld] &  \ar@{-}[r] & \ar@{.}[ld]  \ar@{-}[r] &    \\
\st_2(\wedge^2_B \LL_{B/A})[-2] & \vdots \ar@{-}[d] &  M_2 \otimes_B N_0   & \ar@{.}[ld] &  M_2 \otimes_B N_1  & \ar@{.}[ld] & M_2 \otimes_B N_2 & \ar@{.}[ld]  \cdots  \\
  \vdots    &  & \vdots &    &  \vdots  &   &  \vdots  &       \\
}
\]
In the diagram above, $M_i=\st_i(\wedge^i_B\LL_{B/A})[-i]$, and $N_j=\st_j(\wedge^j_C\LL_{C/B})[-j]$, for $i,j\in \NN$.
Let us explain this diagram: it is describing graded pieces of filtrations
on $\dR_{C/A}$.
Here the rows are representing graded pieces of the Katz--Oda filtration,
and the dotted lines are indicating the Hodge filtration 
(given by things below the dotted line).
Once we take graded pieces with respect to the Hodge filtration, then the vertical
filtration is literally induced by vertical columns, 
starting from left to right, hence the name.

Specializing to the $p$-adic setting, we get the following.

\begin{lemma}
\label{smooth formal Poincare}
Let $A \to B \to C$ be a triangle of $p$-complete flat $\mathbb{Z}_p$-algebras.
Suppose $B/p$ is smooth over $A/p$ of relative equidimension $n$.
Then we have a $p$-adic Katz--Oda filtration on $\dR_{C/A}$ which is strict exact
and witnesses the following sequence:
$$
0 \to \dR_{C/A}^\an \to \dR_{C/B}^\an \xrightarrow{\nabla} \dR_{C/B}^\an
\otimes_B \st_1(\Omega^{1,\an}_{B/A})
\xrightarrow{\nabla} \cdots \xrightarrow{\nabla} 
\dR_{C/B}^\an \otimes_B \st_n(\Omega^{n,\an}_{B/A}) \to 0.
$$
\end{lemma}

Recall that the superscript $(-)^\an$ denotes the derived $p$-completion
of the corresponding objects.
Note that since $\Omega^{i,\an}_{B/A}$ are all finite flat $B$-modules by assumption
and $\dR_{C/B}^\an$ is $p$-complete, the tensor products showing above
are already $p$-complete.

\begin{proof}
Take the derived $p$-completion of the Katz--Oda filtration on $\dR_{C/A}$, 
we get such a strict exact filtration by~\Cref{properties of KO filtration}.
\end{proof}

\subsection{Integral de Rham sheaves}
\label{Integral de Rham sheaves}

For the rest of this section, we focus on the situation spelled out by the following:
\subsubsection*{\textbf{Notation}}
Let $\kappa$ be a perfect field in characteristic $p>0$, 
and let $k=W(\kappa)[\frac{1}{p}]$ be the absolutely unramified discretely valued 
$p$-adic field with the ring of integers $\cO_k=W(\kappa)$.
Fix a separated formally smooth $p$-adic formal schemes $\scrX$ over $\cO_k$.
Denote by $X$ its generic fiber, viewed as an adic space over the Huber pair $(k,\cO_k)$.

In this situation, there is a natural map of ringed sites 
\[
w \colon (X_\pe,\wh\cO_X^+)\rra (\scrX,\cO_\scrX)
\]
which sends an open subset $\scrU\subset \scrX$ to the open subset $U\in X_\pe$, 
where $U$ is the generic fiber of $\scrU$.
This allows us to define inverse image $w^{-1}\cO_\scrX$ of the integral structure sheaf 
$\cO_\scrX$, as a sheaf on the pro-\'etale site $X_\pe$.

On the pro-\'{e}tale site of $X$, we have a morphism of sheaves of $p$-complete $\cO_k$-algebras:
\[
\label{integral triangle}
\tag{\epsdice{1}}
\cO_k \rra w^{-1}\cO_\scrX \rra \wh\cO_X^+.
\]
We refer readers to~\cite[Sections 3 and 4]{Sch13} 
for a detailed discussion surrounding the pro-\'{e}tale
site of a rigid space and structure sheaves on it.
There is a subcategory $X_{\pe/\scrX}^\omega \subset X_\pe$ consisting of affinoid
perfectoid objects $U = \Spa(B,B^+)\in X_\pe$ whose image in $X$ is contained in $w^{-1}(\Spf(A_0))$,
the generic fiber of an affine open $\Spf(A_0) \subset \scrX$.
The class of such objects form a basis for the pro-\'{e}tale topology
by (the proof of)~\cite[Proposition 4.8]{Sch13}.
We first study the behavior of derived de Rham complex for the triangle~\cref{integral triangle}
on $X_{\pe/\scrX}^\omega$.

\begin{proposition}
\label{integral local calculation}
Let $U = \Spa(B,B^+) \in X_\pe$ be an object in $X_{\pe/\scrX}^\omega$,
choose $\Spf(A_0) \subset \scrX$ such that the image of $U$ in $X$ is contained in $w^{-1}(\Spf(A_0))$.
Then
\begin{enumerate}
    \item the natural surjection $\theta \colon A_{inf}(B^+) \twoheadrightarrow B^+$
    exhibits $\dR_{B^+/\cO_k}^\an = A_{\cry}(B^+)$, 
    the $p$-completion of the divided envelope of $A_{inf}(B^+)$
    along $\ker(\theta)$;
    \item the natural surjection $w^\sharp \otimes \theta \colon 
    A_0 \hat{\otimes}_{\cO_k} A_{inf}(B^+) \twoheadrightarrow B^+$
    exhibits $\dR_{B^+/A_0}^\an$ as the $p$-completion of the divided envelope
    of $A_0 \hat{\otimes}_{\cO_k} A_{inf}(B^+)$ along $\ker(w^\sharp \otimes \theta)$;
    \item in both cases, the Hodge filtrations are identified as the $p$-completion
    of PD filtrations;
    \item the filtered algebra $\dR_{B^+/A_0}^\an$ is independent of the choice
    of $A_0$. We denote it as $\dR_{B^+/\scrX}^\an$.
\end{enumerate}
\end{proposition}

\begin{remark}
\label{integral well-defined remark}
In particular, (1) and (2) tells us that these derived de Rham complexes are actually
quasi-isomorphic to an honest algebra viewed as a complex supported on cohomological degree $0$;
(4) tells us that sending $U=\Spa(B,B^+) \in X_{\pe/\scrX}^\omega$ to 
$\dR_{B^+/\scrX}^\an$ gives a well-defined presheaf on $X_{\pe/\scrX}^\omega$.
\end{remark}

\begin{proof}[Proof of~\Cref{integral local calculation}]
Applying~\Cref{dR when relatively perfect}.(4) to the triangles 
$$\cO_k \to A_{inf}(B^+) \to B^+
\text{ and } A_0 \to A_0 \hat{\otimes}_{\cO_k} A_{inf}(B^+) \to B^+$$ 
proves (1) and (2) 
respectively and (3)\footnote{Here we use the unramifiedness of $\cO_k$ to verify the relatively
perfectness assumption in~\Cref{dR when relatively perfect}.}.
As for (4), using separatedness of $\scrX$, we reduce to the situation where
image of $U$ in $X$ is in a smaller open $w^{-1}(\Spf(A_1)) \subset w^{-1}(\Spf(A_0))$.
It suffices to show the natural map $\dR_{B^+/A_0}^\an \to \dR_{B^+/A_1}^\an$ is a
filtered isomorphism, which follows from~\Cref{smooth formal Poincare}
as $A_0/p \to A_1/p$ is \'{e}tale.
\end{proof}

Recall that the subcategory $X_{\pe/\scrX}^\omega \subset X_\pe$ gives a basis
for the topology on $X_\pe$.
Hence any presheaf on $X_{\pe/\scrX}^\omega$ can be sheafified to a sheaf on $X_\pe$.

We define the analytic de Rham sheaf for $\wh\cO_X^+$ over 
$\cO_k$ and $w^{-1}\cO_\scrX$ as follows:
\begin{construction}[$\dR_{\wh\cO_X^+/\cO_k}^\an$ and $\dR_{\wh\cO_X^+/\cO_\scrX}^\an$]
The \emph{analytic de Rham sheaf of} $\wh\cO_X^+/\cO_k$, denoted as $\dR_{\wh\cO_X^+/\cO_k}^\an$,
is the $p$-adic completion of the unfolding
of the presheaf on $X_{\pe/\scrX}^\omega$ which assigns each $U=\Spa(B,B^+)$
the algebra
\(
\dR_{B^+/\cO_k}^\an
\).
We equip it with the decreasing Hodge filtration $Fil^r_\mathrm{H}$ given by 
the image of $p$-completion of the unfolding of the presheaf 
assigning each $U=\Spa(B,B^+)$ the $r$-th Hodge filtration in $\dR_{B^+/\cO_k}^\an$.

The \emph{analytic de Rham sheaf of} $\wh\cO_X^+/\cO_\scrX$, denoted as
$\dR_{\wh\cO_X^+/\cO_\scrX}^\an$, is the $p$-adic completion of the unfolding
of the presheaf on $X_{\pe/\scrX}^\omega$ which assigns each $U=\Spa(B,B^+)$
the filtered algebra
\(
\dR_{B^+/\scrX}^\an
\).
Similarly we equip it with the decreasing Hodge filtration
$Fil^r_H$ given by the image of $p$-completion of the unfolding of the presheaf
whose value on each $U=\Spa(B,B^+)$ is the $r$-th Hodge filtration in 
$\dR_{B^+/\scrX}^\an$.
\end{construction}

The fact that these definitions/constructions make sense follows 
from~\Cref{integral local calculation} and~\Cref{integral well-defined remark}.

One may also define the corresponding mod $p^n$ version of these sheaves.
Since sheafifying commutes with arbitrary colimit,
the $p$-adic completion of the sheafification of a presheaf $F$
is the same as the inverse limit over $n$ of the sheafification of 
presheaves $F/p^n$.
Therefore we have $\dR_{\wh\cO_X^+/\cO_k}/p^n$ is the same as the sheafification
of the presheaf $\dR_{B^+/\cO_k}/p^n$.
Its $r$-th Hodge filtration agrees with the sheafification of the presheaf 
$Fil^r_H(\dR_{B^+/\cO_k}/p^n)$, as sheafifying is an exact functor.
Similar statements can be made for the mod $p^n$ version of
$\dR_{\wh\cO_X^+/\cO_\scrX}^\an$ and its Hodge filtrations.

Now the strict exact Katz--Oda filtration obtained in the~\Cref{smooth formal Poincare}
gives us the following:
\begin{corollary}[Crystalline Poincar\'{e} lemma]
\label{Crystalline Poincare sequence}
There is a functorial $\dR_{\wh\cO_X^+/\cO_k}^\an$-linear 
strict exact sequence of filtered sheaves on $X_\pe$:
\[
0 \to \dR_{\wh\cO_X^+/\cO_k}^\an \to \dR_{\wh\cO_X^+/\cO_\scrX}^\an
\xrightarrow{\nabla} \dR_{\wh\cO_X^+/\cO_\scrX}^\an \otimes_{w^{-1}\cO_\scrX }
\st_1(w^{-1}\Omega^{1,\an}_\scrX) \xrightarrow{\nabla} \cdots
\]
\[
\cdots \xrightarrow{\nabla} 
\dR_{\wh\cO_X^+/\cO_\scrX}^\an \otimes_{w^{-1}\cO_\scrX }
\st_d(w^{-1}\Omega^{d,\an}_\scrX) \to 0,
\]
where $d$ is the relative dimension of $\scrX/\cO_k$.
\end{corollary}

\begin{proof}
Using the discussion before this Corollary, we reduce to checking
this at the level of presheaves on $X_{\pe/\scrX}^\omega$.
Since now everything in sight are supported cohomologically in degree $0$
with filtrations given by submodules because 
of~\Cref{integral local calculation},
the strict exact Katz--Oda filtration in~\Cref{smooth formal Poincare} implies what we want.
\end{proof}

\begin{remark}
\label{integral separated remark}
We can drop the separatedness assumption on $\scrX$ as follows.
Since any formal scheme is covered by affine ones, and affine formal schemes are
automatically separated,
we may define all these de Rham sheaves on each slice subcategory
of the pro-\'{e}tale site of the rigid generic fiber of affine opens of $\scrX$.
Similar to the proof of~\Cref{integral local calculation}.(4),
we can show these de Rham sheaves satisfy the base change formula with respect
to maps of affine opens of $\scrX$
(by appealing to~\Cref{smooth formal Poincare} again),
hence these sheaves on the slice subcategories glue to a global one.
The Crystalline Poincar\'{e} lemma obtained above holds verbatim
as exactness of a sequence of sheaves may be checked locally.
\end{remark}

\subsection{Comparing with Tan--Tong's crystalline period sheaves}
Lastly we shall identify the two de Rham sheaves defined 
above with two period sheaves
that show up in the work of Tan--Tong~\cite{TT19}.
We refer readers to Definitions 2.1.~and 2.9.~of loc.~cit.~for the
meaning of period sheaves $\mathbb{A}_{\cry}$ and $\cO\mathbb{A}_{\cry}$
and their PD filtrations.


We look at the triangle of sheaves of rings:
\[
\cO_k \to w^{-1}(\cO_\scrX) \hat{\otimes}_{\cO_k} \mathbb{A}_{inf} 
\xrightarrow{w^\sharp \hat{\otimes} \theta} \wh\cO_X^+.
\]

\begin{theorem}
\label{crystalline period comparison}
The triangle above induces a filtered isomorphism of sheaves:
$\dR_{\wh\cO_X^+/\cO_k} \cong \mathbb{A}_{\cry}$
and
$\dR_{\wh\cO_X^+/\cO_\scrX} \cong \cO\mathbb{A}_{\cry}$.

Moreover, under this identification, 
the Crystalline Poincar\'{e} sequence in~\Cref{Crystalline Poincare sequence}
agrees with the one obtained in~\cite[Corollary 2.17]{TT19}.
\end{theorem}

\begin{proof}
We check these isomorphisms modulo $p^n$ for any $n$.
For both cases, the de Rham sheaf and
the crystalline period sheaf are both
unfoldings of the same PD envelope presheaf (with its PD filtrations)
on $X_{\pe/\scrX}^\omega$:
for the de Rham sheaves this statement follows 
from~\Cref{integral local calculation} and base change formula of PD envelope
(note that taking PD envelope is a left adjoint functor, hence commutes with
colimit, in particular, it commutes with modulo $p^n$ for any $n$),
for the crystalline period sheaf this follows from the definition
(note that although the $\cO\mathbb{A}_{inf}$ defined in Tan--Tong's work
uses uncompleted tensor of $w^{-1}(\cO_\scrX)$ and $\mathbb{A}_{inf}$
instead of the completed tensors we are using here,
the difference goes away when we modulo any power of $p$ and restricts
to the basis of affinoid perfectoid objects).

Therefore for both cases,
we have natural isomorphisms modulo $p^n$ for any $n$,
taking inverse limit gives the result we want
as all sheaves are $p$-adic completion of their modulo $p^n$ versions.

The claim about matching Poincar\'{e} sequences follows by unwinding definitions.
Indeed we need to check that $\nabla$ defined in these two sequences
agree, but since $\nabla$ is linear over 
$\dR_{\wh\cO_X^+/\cO_k} \cong \mathbb{A}_{\cry}$,
it suffices to check that $\nabla$ agrees on $u_i$
which is the image of $T_i - S_i$ 
(notation from loc.~cot.~and~\Cref{dR of perfectoid circle} respectively)
by functoriality of the Poincar\'{e} sequence.
One checks that in both cases their image under $\nabla$ is $1 \otimes dT_i$.
\end{proof}

\section{Rational Theory}
\label{Rational Theory}

For the rest of this article, we shall study a rational version
of the previous derived de Rham complex.
Let us spell out the setup by recalling the following notation: $k$ is a $p$-adic field
with ring of integers denoted by $\cO_k$ and $X$ is a
separated\footnote{Just like~\Cref{integral separated remark} suggests, 
we can remove the separatedness assumption in the end.} rigid space
over $k$ which we view as an adic space over $\Spa(k,\cO_k)$.

\subsection{Affinoid construction}
In this subsection, we recall the construction of the analytic cotangent complex and
give the construction of the analytic derived de Rham complex, 
for a map of Huber rings over a $k$.
For a detailed discussion of the analytic cotangent complex 
(for topological finite type algebras), we refer readers to \cite[Section 7.1-7.3]{GR03}.

Let $f \colon (A,A^+)\ra (B,B^+)$ be a map of complete Huber rings over $k$. 
Denote by $\cC_{B/A}$ the filtered category of pairs $(A_0,B_0)$, 
where $A_0$ and $B_0$ are rings of definition of $(A,A^+)$ and $(B,B^+)$ separately, 
such that $f(A_0)\subset B_0$.

\begin{construction}[Analytic cotangent complex, affinoid]
\label{cot construction aff}
For each $(A_0,B_0) \in \cC_{B/A}$, denote by $\LL_{B_0/A_0}^\an$ 
the integral analytic cotangent complex of $A_0 \ra B_0$ as in the~\Cref{int construction}.
The \emph{analytic cotangent complex of $f \colon (A,A^+)\ra (B,B^+)$}, denoted by $\LL_{B/A}^\an$,
is defined as the filtered colimit
\[
\LL_{B/A}^\an \coloneqq \underset{(A_0,B_0)\in \cC_{B/A}}{\colim} 
\LL_{B_0/A_0}^\an[\frac{1}{p}].
\]
\end{construction}

For the convenience of readers, let us list a few properties of analytic cotangent complex
for a morphism of rigid affinoid algebras obtained by Gabber--Romero.

\begin{theorem}
\label{GR statements}
Let $A \to B$ be a morphism of $k$-affinoid algebras, then we have:
\begin{enumerate}
\item~\cite[Theorem 7.1.33.(i)]{GR03} 
$\LL_{B/A}^\an$ is in $\scrD^{\leq 0}(B)$ and is pseudo-coherent over $B$;
\item~\cite[Lemma 7.1.27.(iii) and Equation 7.2.36]{GR03} 
the $0$-th cohomology of the analytic cotangent complex is given by
the analytic relative differential:
$\rmH_0(\LL_{B/A}^\an) \simeq \Omega^\an_{B/A}$;
\item~\cite[Theorem 7.2.42.(ii)]{GR03} if $A \to B$ is smooth, then 
$\LL^\an_{B/A} \simeq \Omega^\an_{B/A}[0]$;
\item~\cite[Lemma 7.2.46.(ii)]{GR03} 
if $A \to B$ is surjective, then the analytic cotangent complex agrees with the classical cotangent complex:
$\LL_{B/A} \simeq \LL^\an_{B/A}$.
\end{enumerate}
\end{theorem}
	
\begin{construction}[Analytic derived de Rham complex, affinoid]
\label{ddR construction aff}
Let $f \colon (A,A^+) \ra (B,B+)$ be a map of complete Huber rings over $k$.
For each $(A_0,B_0)\in \cC_{B/A}$, 
by the \Cref{int construction} we could define the integral analytic derived de Rham complex $\dR_{B_0/A_0}^\an$, 
as an object in $\CAlg(\DF(A_0))$.
Then the \emph{analytic derived de Rham complex $\dR_{B/A}^\an$} of 
$(B,B^+)$ over $(A,A^+)$, as an object in $\CAlg(\DF(A))$, 
is defined to be the filtered colimit
\[
\dR_{B/A}^\an \coloneqq \underset{(A_0,B_0)\in \cC_{B/A}}{\colim}\dR_{B_0/A_0}^\an[\frac{1}{p}].
\]

Moreover, the \emph{(Hodge) completed analytic derived de Rham complex $\wh\dR_{B/A}^\an$}
of $(B,B^+)$ over $(A,A^+)$, as an object in $\CAlg(\wh\DF(A))$, 
is defined as the derived filtered completion of $\dR_{B/A}^\an$.
\end{construction}

By the construction, the graded pieces of the filtered complete $A$-complex $\wh\dR_{B/A}^\an$ is given by
\begin{align*}
\tag{\epsdice{2}}
\label{graded pieces of ddR}
\Gr^i(\wh\dR_{B/A}^\an)&\cong \underset{(A_0,B_0)\in \cC_{B/A}}{\colim}\Gr^i(G(A_0,B_0))\\
&\cong \underset{(A_0,B_0)\in \cC_{B/A}}{\colim}(\rmL\wedge^i\LL_{B_0/A_0}^\an[\frac{1}{p}])[-i]\\
&\cong (\rmL\wedge^i\LL_{B/A}^\an)[-i],
\end{align*}
due to the fact that the functor $\Gr^i$ preserves filtered colimits.

\begin{remark}[Complexity of the construction]
The two rational constructions above involve colimits among all rings of definitions and seem to be very complicated.
A naive attempt would be taking the usual cotangent/derived de Rham complex of $A^+\ra B^+$, 
apply the derived $p$-adic completion and invert $p$ 
(and do the filtered completion, for the derived de Rham complex case) directly.
This would \emph{not} give us the expected answer in general, 
which is essentially due to the possible existence of nilpotent elements in $(A,A^+)$ and $(B,B^+)$.

Take the map $(k,\cO_k)\ra (B,B^+)$ for $B=k\langle \epsilon\rangle/(\epsilon^2)$ as an example.
Then a ring of definition $B_0$ of $B$ could be $\cO_k\langle \epsilon\rangle/(\epsilon^2)$, 
while there is only one open integral subring of $B$ that contains $\cO_k$, 
namely $\cO_k \oplus k\cdot \epsilon$.
In this case, it is easy to see that the derived $p$-completion of cotangent complexes 
$\LL_{B^+/\cO_k}$ and $\LL_{B_0/\cO_k}$ are different,
and remain so after inverting $p$.
\end{remark}

\begin{remark}[Simplified construction for uniform Huber pairs]
\label{sim construction}
Assume both of the Huber pairs $(A,A^+)\ra (B,B^+)$ are \emph{uniform}; 
namely the subrings of power bounded elements $A^\circ$ and $B^\circ$ are bounded in $A$ and $B$ separately.
Then both $A^+$ and $B^+$ are rings of definition of $A$ and $B$ separately.
In particular, the~\Cref{cot construction aff} and the~\Cref{ddR construction aff} can be simplified as follows:
\begin{align*}
\LL_{B/A}^\an&=\LL_{B^+/A^+}^\an[\frac{1}{p}],\\
\wh\dR_{B/A}^\an&=filtered~completion~of~((derived~p-completion~of~\dR_{B^+/A^+})[\frac{1}{p}]),
\end{align*}
where we recall that $\LL_{B^+/A^+}^\an$ is the derived 
$p$-completion of the classical cotangent complex $\LL_{B^+/A^+}$, 
and $\dR_{B^+/A^+}$ is the classical derived de Rham complex of $B^+/A^+$, 
as in~\cite[Examples 5.11-5.12]{BMS2}.

Examples of uniform Huber pairs include
reduced affinoid algebras over discretely valued or algebraically closed non-Archimedean
fields~\cite[Theorem 3.5.6]{FvdP},
and perfectoid affinoid algebras~\cite[Theorem 6.3]{Sch12}.
\end{remark}

An arithmetic example of the Hodge-completed analytic derived de Rham complex
has been worked out by Beilinson.

\begin{example}[{\cite[Proposition 1.5]{Bei12}}]
\label{Beilinson Colmez example}
We have a filtered isomorphism:
\[
B_{\dR}^+ \cong \wh\dR^\an_{\overline{\mathbb{Q}_p}/\mathbb{Q}_p}.
\]
\end{example}

Next we work out a geometric example.
Let us compute
the Hodge-completed analytic derived de Rham complex of a
perfectoid torus over a rigid analytic torus.
Following the notation in~\Cref{cotangent complex of perfectoid circle},
let $R = \mathbb{Z}_p \langle T_1^{\pm 1}, \ldots, T_n^{\pm 1} \rangle$, and 
$R_{\infty} = \mathbb{Z}_p \langle T_1^{\pm 1/p^{\infty}}, 
\ldots, T_n^{\pm 1/p^{\infty}} \rangle = 
R \langle S_1^{1/p^{\infty}}, \ldots, S_n^{1/p^{\infty}} \rangle/(T_i - S_i; 
1 \leq i \leq n)$.
\begin{example}
\label{rational torus example}
Continue with~\Cref{dR of perfectoid circle}.
After inverting $p$ and completing along Hodge filtrations, we see that
$\wh\dR^\an_{R_{\infty}[1/p]/R[1/p]}$ is given by the completion of
$\mathbb{Q}_p \langle T_i^{\pm 1}, S_i^{1/p^\infty} \rangle$ along $\{T_i-S_i; 1 \leq i \leq n\}$.
Here we use~\Cref{sim construction} to relate
$\wh\dR^\an_{R_{\infty}/R}$ and $\wh\dR^\an_{R_{\infty}[1/p]/R[1/p]}$.
A more explicit presentation is
\[\wh\dR^\an_{R_{\infty}[1/p]/R[1/p]} = 
\mathbb{Q}_p \langle S_1^{\pm 1/p^\infty}, \ldots, S_n^{\pm 1/p^\infty} \rangle [\![X_1, \ldots, X_n]\!]
\]
via change of variable $T_i = X_i + S_i$ (hence $T_i^{-1} = S_i^{-1} \cdot (1 + S_i^{-1}X_i)^{-1}$),
c.f.~the notation before \cite[Proposition 6.10]{Sch13}.
\end{example}

We need to understand the output of these constructions for general perfectoid affinoid
algebras relative to affinoid algebras.
The following tells us that in this situation, the Hodge completed analytic derived de Rham complex
can be computed with any ring of definition inside the affinoid algebra.

\begin{lemma}
Let $(A,A^+)$ be a topologically finite type complete Tate ring over $(k,\cO_k)$,
with $A_0 \subset A^+$ being a ring of definition.
Let $(B,B^+)$ be a perfectoid algebra over $(A,A^+)$.
Then we have:
\begin{enumerate}
\item The analytic cotangent complex $\LL_{B/A}^\an \cong \LL_{B^+/A_0}^\an[1/p]$.
\item The Hodge completed analytic derived de Rham complex 
$\wh\dR_{B/A}^\an \cong \wh{\dR_{B^+/A_0}^\an[1/p]}$, 
where the latter is the Hodge completion of $\dR_{B^+/A_0}^\an[1/p]$.
\end{enumerate}
\end{lemma}

In the proof below we will show a stronger statement: 
the transition morphisms of the colimit process computing left hand side
in~\Cref{cot construction aff} and~\Cref{ddR construction aff} are all isomorphisms.

\begin{proof}
Let $A_0' \subset A^+$ be another ring of definition containing $A_0$.
It suffices to show that $\LL_{B^+/A_0}^\an[1/p] \cong \LL_{B^+/A_0'}^\an[1/p]$
and similarly for their Hodge completed analytic derived de Rham complexes.
Since Hodge completed analytic derived de Rham complex of both sides are derived complete with respect to
the Hodge filtration, whose graded pieces, by~\Cref{graded pieces of ddR}, are
derived wedge product of relevant analytic cotangent complexes,
we see that the statement about Hodge completed analytic derived de Rham complex
follows from the statement about analytic cotangent complex.

To show $\LL_{B^+/A_0}^\an[1/p] \cong \LL_{B^+/A_0'}^\an[1/p]$,
we appeal to the fundamental triangle of (analytic) cotangent complexes:
\[
\LL_{A_0'/A_0}^\an {\otimes_{A_0'}} B^+ \rra \LL_{B^+/A_0}^\an \rra \LL_{B^+/A_0'}^\an.
\]
Here the tensor product does not need an extra $p$-completion as $\LL_{A_0'/A_0}$ is pseudo-coherent, 
see~\cite[Theorem 7.1.33]{GR03}.
By~\cite[Theorem 7.2.42]{GR03},
the $p$-complete cotangent complex $\LL_{A_0'/A_0}^\an$ satisfies
\[
\LL_{A_0'/A_0}^\an[\frac{1}{p}]=\Omega_{A_0'[\frac{1}{p}]/A_0[\frac{1}{p}]}^{1, \an},
\]
which vanishes as $A_0'[\frac{1}{p}]$ and $A_0[\frac{1}{p}]$ are both equal to $A$.
Therefore the natural map 
\[
\LL_{B^+/A_0}^\an[\frac{1}{p}]\rra \LL_{B^+/A_0'}^\an[\frac{1}{p}]
\]
induced by $A_0 \ra A_0'$ is a quasi-isomorphism.
\end{proof}

We can understand the associated graded algebra of analytic de Rham complex of perfectoid affinoid
algebras over affinoid algebras via the following~\Cref{ddR aff computation}.
Let $K$ be a perfectoid field extension of $k$ that contains $p^n$-roots of unity for all $n\in \NN$.

\begin{theorem}
\label{ddR aff computation}
Let $(A,A^+)$ be a topologically finite type complete Tate ring over $(k,\cO_k)$.
Assume $(B,B^+)$ is a perfectoid algebra containing both $(K,\cO_K)$ and $(A,A^+)$.
Then the graded algebra $\Gr^\ast(\wh\dR_{B/A}^\an)$ 
admits a natural graded quasi-isomorphism to the derived divided power algebra 
$\rmL\Gamma_B^\ast(\Gr^1(\wh\dR_{B/A}^\an))$, 
where the first graded piece fits into a distinguished triangle:
\[
B(1) \rra \Gr^1(\wh\dR_{B/A}^\an) \cong \LL_{B/A}^\an[-1] \rra B\otimes_A \LL_{A/k}^\an,
\]
which is functorial in $(B,B^+)/(A,A^+)$.
In particular, the graded pieces are $B$-pseudo-coherent.
\end{theorem}

Here $B(1)$ denote $\ker(\theta)/\ker(\theta)^2$
where $\theta \colon A_{inf}(B^+)[1/p] \twoheadrightarrow B$ is Fontaine's $\theta$ map.
Our assumption of $(B,B^+)$ containing $(K,\cO_K)$ ensures that this is
(non-canonically) isomorphic to $B$ itself, see~\cite[Lemma 6.3]{Sch13}.
After sheafifying everything, it corresponds to a suitable Tate twist of $B$.

\begin{proof}
The identification $\Gr^1(\wh\dR_{B/A}^\an) \cong \LL_{B/A}^\an[-1]$
is already spelled out by~\Cref{graded pieces of ddR}.

Let us fix a single choice of pair of rings of definition $(A_0,B^+)$ in $\cC_{B/A}$.
Here $A_0$ is topologically finitely presented over $\cO_k$, 
and $B^+$ contains $\cO_K$ for 
$K$ a perfectoid field containing all $p^n$-th roots of unity.

Consider the following triple:
\(
\cO_k \rra A_0 \rra B^+,
\)
it induces the following triangle
\[
\LL_{A_0/\cO_k}^\an \otimes_{A_0} B^+ \rra \LL_{B^+/\cO_k}^\an \rra \LL_{B^+/A_0}^\an.
\]
Here we again have used the pseudo-coherence~\cite[Theorem 7.1.33]{GR03} of $\LL_{A_0/\cO_k}^\an$.
We need to show $\LL_{B^+/\cO_k}^\an[1/p] \cong B(1)[1]$.
To that end, let $W$ be the Witt ring of the residue field of $\cO_k$.
By looking at the triple $W \to \cO_k \to B^+$, we get another sequence
\[
\LL_{B^+/W}^\an \cong B^+(1)[1] \rra \LL_{B^+/\cO_k}^\an \rra \LL_{\cO_k/W}^\an \otimes_{\cO_k} B^+[1],
\]
where the first identification follows from~\Cref{integral local calculation},
and the tensor product does not an extra completion again by coherence of $\LL_{\cO_k/W}^\an$.
Since $k/W[1/p]$ is finite \'{e}tale, we conclude that $\LL_{\cO_k/W}^\an[1/p] = 0$
by~\cite[Theorem 7.2.42]{GR03}.
This ends the proof of the structure of $\LL_{B/A}^\an$.

Now we turn to the higher graded piece.
The $i$-th graded pieces $\Gr^i(\wh\dR_{B/A}^\an)$ is quasi-isomorphic to
$(\rmL\wedge^i \LL_{B/A}^\an)[-i]$, which by rewriting in terms of the first graded piece is
\[
(\rmL\wedge^i(\Gr^1(\wh\dR_{B/A}^\an)[1]))[-i].
\]
So by the relation between the derived wedge product and the derived divided power funcotr 
(with bounded above input, see~\cite[V.4.3.5]{Ill71}), we get
\[
\Gr^i(\wh\dR_{B/A}^\an)\cong \rmL\Gamma_B^i(\Gr^1(\wh\dR_{B/A})),
\]
and we get the divided power algebra structure of the graded algebra $\Gr^\ast(\wh\dR_{B/A}^\an)$.
\end{proof}

Consequently we get cohomological bounds for perfectoid affinoid algebras over
various types of affinoid algebras.
The notion of local complete intersection and embedded codimension
(in the situation that we are working with) is discussed in the Appendix.

\begin{corollary}
\label{cohomological bound aff}
Let $(B,B^+)/(A,A^+)$ be as in the statement of~\Cref{ddR aff computation}.
Then we have
\begin{enumerate}
    \item $\wh\dR_{B/A}^\an \in \scrD^{\leq 0}(A)$;
    \item if $A/k$ is smooth, then $\wh\dR_{B/A}^\an \in \scrD^{[0,0]}(A)$;
    \item if $A/k$ is local complete intersection with embedded codimension $c$,
    then $\wh\dR_{B/A}^\an \in \scrD^{[-c,0]}(A)$.
\end{enumerate}
\end{corollary}

\begin{proof}
Since the out put of $\wh\dR^\an$ is always derived complete with respect to its Hodge filtration,
it suffices to show these statements for the graded pieces of Hodge filtration.

For (1), this follows from the fact that $\LL_{B/A}^\an \in \scrD^{\leq 0}(B)$.
(2) follows from (3) as smooth affinoid algebra has embedded codimension $0$.

As for (3), we check the graded pieces of Hodge filtration in this case is in $\scrD^{[-c,0]}$.
In fact, we shall show that the graded pieces, as objects in $\scrD(B)$, 
have Tor amplitude $[-c,0]$.
First since $B$ contains $\mathbb{Q}$, we have
\[
\Gr^i(\wh\dR_{B/A}^\an) \cong \rmL\Gamma_B^i(\Gr^1(\wh\dR_{B/A}^\an))
\cong \rmL\Symm_B^i(\Gr^1(\wh\dR_{B/A}^\an)).
\]
Using the triangle in~\Cref{ddR aff computation}, it suffices to show
$\rmL\Symm_B^j(B\otimes_A \LL_{A/k}^\an)$ have Tor amplitude $[-c,0]$ for all $j$.
Since $\rmL\Symm_B^j(B\otimes_A \LL_{A/k}^\an) \cong B \otimes_A \rmL\Symm_A^j(\LL_{A/k}^\an)$,
we are done by~\Cref{embedded codimension control}.
\end{proof}

\subsection{Poincar\'{e} sequence}

In this subsection we explain the Poincar\'{e} sequence for Hodge completed 
de Rham complexes.

\begin{lemma}
\label{algebraic Poincare mod Filn}
Let $B \to C$ be an $A$-algebra morphism.
Then for every $j \in \mathbb{N}$, the Katz--Oda filtration on $\dR_{C/A}$
induces a functorial strict exact filtration on $\dR_{C/A}/\Fil^j_{\mathrm{H}}$,
witnessing the following sequence:
\[
\dR_{C/A}/\Fil^j \to \dR_{C/B}/\Fil^j \xrightarrow{\nabla} \dR_{C/B}/\Fil^{j-1}
\otimes_B \st_1(\LL_{B/A})
\xrightarrow{\nabla} \cdots \xrightarrow{\nabla} 
\dR_{C/B}/\Fil^1 \otimes_B \st_{j-1}(\rmL\wedge^{j-1}\LL_{B/A}).
\]
Here $\dR_{C/A}$ and $\dR_{C/B}$ are equipped with Hodge filtrations.

Moreover $\Fil^i_{\mathrm{KO}}(\dR_{C/A}/\Fil^j_{\mathrm{H}}) = 0$ whenever $i > j$.
\end{lemma}

\begin{proof}
We consider the induced Katz--Oda filtration on
$\dR_{C/A}/\Fil^i_{\mathrm{H}}$. 
Since we have mod out Hodge filtration, the~\Cref{properties of KO filtration} (3)
implies the desired vanishing of the $\Fil^i_{\mathrm{KO}}$ when $i > j$,
and this in turn implies the strict exactness of these filtrations.
\end{proof}

Specializing to the $p$-adic situation, we get the following:
\begin{lemma}
\label{padic Poincare mod Filn}
Let $(A,A^+) \to (B,B^+) \to (C,C^+)$ be a triangle of complete Huber rings
over $k$. Then for each $j \in \NN$,
we have a functorial strict exact filtration on $\dR_{C/A}^\an/\Fil^j$,
still denoted by $\Fil^i_{\mathrm{KO}}$, witnessing the following sequence:
\[
\dR_{C/A}^\an/\Fil^j \to \dR_{C/B}^\an/\Fil^j \xrightarrow{\nabla} 
\dR_{C/B}^\an/\Fil^{j-1}
\wh\otimes_B \st_1(\LL_{B/A}^\an)
\xrightarrow{\nabla} \cdots \xrightarrow{\nabla} 
\dR_{C/B}^\an/\Fil^1 \wh\otimes_B \st_{j-1}(\rmL\wedge^{j-1}\LL_{B/A}^\an).
\]
Here $\dR_{C/A}^\an/\Fil^j$ and $\dR_{C/B}^\an/\Fil^j$ are equipped with Hodge filtrations.

Moreover $\Fil^i_{\mathrm{KO}}(\dR_{C/A}^\an/\Fil^j) = 0$ whenever $i > j$.
\end{lemma}

\begin{proof}
For any triangle of rings of definition $A_0 \to B_0 \to C_0$,
we $p$-complete the filtration from~\Cref{algebraic Poincare mod Filn}
and invert $p$, 
then we take the colimit over all triangles of such triples of rings of definition
to get the filtration sought after.
Since all the operations involved are (derived-)exact, 
the resulting filtration still has vanishing: $\Fil^i_{\mathrm{KO}} = 0$ whenever $i > j$,
and this again implies the strict exactness.
\end{proof}

In the setting of the above Lemma,
after taking limit with $j$ going to $\infty$, 
we get the following:
\begin{corollary}[Poincare Lemma]
\label{Poincare Lemma}
Let $(A,A^+) \to (B,B^+) \to (C,C^+)$ be a triangle of complete Huber rings
over $k$.
Then there is a functorial strict exact filtration on $\wh\dR_{C/A}^\an$
witnessing the following sequence
\[
\label{true Poincare sequence}
\tag{\epsdice{3}}
\wh\dR_{C/A}^\an \rra \wh\dR_{C/B}^\an \xrightarrow{\nabla} 
\wh\dR_{C/B}^\an \hat{\otimes}_B \st_1(\LL_{B/A}^\an) \to \cdots.
\]

The $\nabla$'s are $\wh\dR_{C/A}^\an$-linear and satisfy Newton--Leibniz rule.
\end{corollary}

\begin{proof}
Take limit in $j$ of the Katz--Oda filtrations on $\dR_{C/A}^\an/\Fil^j$
in~\Cref{padic Poincare mod Filn} gives the desired filtration.
Indeed, inverse limit of complete filtrations is again complete.
Moreover we have
\[
\Gr^i_{\mathrm{KO}}(\wh\dR_{C/A}^\an) \cong \lim_j \Gr^i_{\mathrm{KO}}(\dR_{C/A}^\an/\Fil^j)
\cong \lim_j \left(\dR_{C/B}^\an/\Fil^{j-i} 
\wh\otimes_B \st_{i}(\rmL\wedge^{i}\LL_{B/A}^\an)[-i]\right)
\cong \wh\dR_{C/B}^\an \wh\otimes_B \st_{i}(\rmL\wedge^{i}\LL_{B/A}^\an)[-i],
\]
so we get the statement about the sequence
that this filtration is witnessing.

Lastly the statement about $\nabla$ is the consequence of a general statement about
multiplicative filtrations on $E_\infty$-algebras, 
see the proof of~\Cref{properties of KO filtration} (2).
\end{proof}

\begin{remark}
In fact, the discussion of the Poincar'e sequence above could be obtained via a product formula
\[
\wh\dR_{C/A} \wh\otimes_{\wh\dR_{B/A}} B \cong \wh\dR_{C/B},
\]
similar to the discussion in subsection \Cref{ddR for a triple}.
Here the formula can be obtained via a filtered completion, by $p$-completing the formula
in~\Cref{ddR base change formula} and inverting $p$.

We mention that this formula could also be proved by applying the symmetric monoidal functor
$\Gr^*$ and checking the graded pieces, where the claim is reduced to the distinguished
triangle of analytic cotangent complexes for a triple of Huber pairs.
\end{remark}

\subsection{Rational de Rham sheaves}
In this subsection, we shall apply the construction of the 
(Hodge completed) analytic derived de Rham complexes to the triangle of sheaves of Huber rings 
$(k,\cO_k) \ra (\nu^{-1}\cO_{X}, \nu^{-1}\cO_{X}^+)\ra (\wh\cO_X, \wh\cO_X^+)$
on the pro-\'etale site,
where $\nu \colon X_{\pe} \to X$ is the standard map of sites.
The procedure is similar to what we did in~\Cref{Integral de Rham sheaves},
except now we allow $X$ to be locally complete 
intersection~\footnote{See Appendix for the notion of local complete intersection that we are using here.}
over $k$, and we shall use the unfolding
as discussed in~\Cref{unfolding}.

Let $K$ be a perfectoid field extension of $k$ that contains $p^n$-roots of unity for all $n\in \NN$.
There is a subcategory $X_{\pe}^\omega \subset X_\pe$ consisting of affinoid
perfectoid objects $U = \Spa(B,B^+) \in X_{K,\pe}$ whose image in $X$ is contained in
an affinoid open $\Spa(A,A^+) \subset X$.
The class of such objects form a basis for the pro-\'{e}tale topology
by (the proof of)~\cite[Proposition 4.8]{Sch13}.

\begin{proposition}
\label{hypersheaf properties}
Let $U = \Spa(B,B^+) \in X_{\pe}^\omega$,
choose $\Spa(A,A^+) \subset X$ such that the image of $U$ in $X$ is contained in $\Spa(A,A^+)$.
Then
\begin{enumerate}
    \item the natural surjection $\theta \colon A_{inf}(B^+)[1/p] \twoheadrightarrow B$
    exhibits $\wh\dR_{B/k}^\an = B_{\dR}^+(B)$, 
    and the Hodge filtrations are identified with the $\ker(\theta)$-adic filtrations;
    \item the presheaf defined by sending $U$ to 
    $\Gr^i(\wh\dR_{B/k}^\an)$ is a hypersheaf;
    \item the assignment sending $U$ to $\dR_{B/A}^\an/\Fil^n$ is independent of the choice
    of $\Spa(A,A^+)$, hence so is the assignment sending $U$ to $\wh\dR_{B/A}^\an$,
    we denote it as $\wh\dR_{B/X}^\an$;
    \item assuming $X/k$ is a local complete intersection, then the presheaf assigning
    $U$ to $\Gr^i(\wh\dR_{B/X}^\an)$ is a hypersheaf.
\end{enumerate}
\end{proposition}

\begin{proof}
(1) and (3) follows from the same proof of~\Cref{integral local calculation}~(1) and (4) respectively.

Now we prove (2). 
The $i$-th graded piece of $\wh\dR_{B/k}^\an$ is isomorphic to $B(i)$ 
by~\Cref{ddR aff computation} (with $(A,A^+)$ there being $(k,\cO_k)$).
These are hypersheaves as they are supported in cohomological degree $0$ and satisfy higher acyclicity
by~\cite[Lemma 4.10]{Sch13}.

Lastly we we turn to (4). 
The graded pieces of $\wh\dR_{B/X}^\an$, by (2), is the same as $\wh\dR_{B/A}^\an$
for any choice of $A$.
Notice that, by~\Cref{ddR aff computation},
the $\Gr^i(\dR_{B/A}^\an)$ has a finite step filtration with graded pieces given by
$(\rmL\wedge^j\LL_{A/k}^\an) \otimes_A B(i-j)$.
Since hypersheaf property satisfies two-out-of-three principle in a triangle,
it suffices to show that the assignment sending 
\[
\Spa(B,B^+)=U \mapsto \left(\rmL\wedge^j\LL_{A/k}^\an\right) \otimes_A B(i-j)
\]
is a hypersheaf. 
This follows from the fact that $\LL_{A/k}^\an$
is a perfect complex (as $X$ is assumed to be a local complete intersection over $k$) and, again,
that sending $U$ to $B(m)$ is a hypersheaf for any $m \in \mathbb{Z}$.
\end{proof}

In particular,~\Cref{hypersheaf properties} tells us that the presheaves given by 
\[
\Spa(B,B^+)= U \in X_{\pe}^\omega \mapsto 
\begin{cases}
\wh\dR_{B/k}^\an/\Fil^n \text{ or }\\
\wh\dR_{B/k}^\an \text{ or }\\
\wh\dR_{B/X}^\an/\Fil^n \text{ or }\\
\wh\dR_{B/X}^\an
\end{cases},
\]
are all hypersheaves on $X_{\pe}^\omega$ 
(assuming $X/k$ is a local complete intersection for the latter two), 
using the fact that the hypersheaf property is preserved under taking limit, 
so we may unfold them to get a hypersheaf on $X_\pe$.

The authors believe that the conclusion of~\Cref{hypersheaf properties} (4)
(or a variant) should still hold for general rigid spaces instead of
only the local complete intersection ones.
Hence we ask the following:
\begin{question}
Given any rigid space $X/k$, is it true that
the presheaf assigning $U$ to $\Gr^i(\wh\dR_{B/X}^\an)$
is always a hypersheaf?
\end{question}

The subtlety is that a pro-\'{e}tale map of affinoid perfectoid algebras
need not be flat. 

Now we are ready to define the hypersheaf version of the relative de Rham cohomology.

\begin{definition}
The \emph{Hodge-completed analytic derived de Rham complex} of $X_{\pe}$ over $k$,
denoted by $\wh\dR_{X_\pe/k}^\an$,
is defined to be the unfolding of the hypersheaf on $X_{\pe}^\omega$
whose value at $U = \Spa(B,B^+) \in X_{\pe}^\omega$ is $\wh\dR_{B/k}^\an$.

Similarly we define a filtration on $\wh\dR_{X_\pe/k}^\an$ by unfolding
the Hodge filtration on $\wh\dR_{B/k}^\an$.
Since values of unfolding are computed by derived limits,
we see immediately that $\wh\dR_{X_\pe/k}^\an$ is derived complete
with respect to the filtration.
\end{definition}

This construction is related to Scholze's period sheaf $\mathbb{B}_{\dR}^+$ 
(see~\cite[Definition 6.1.(ii)]{Sch13})
by the following:
\begin{proposition}
\label{BdR+ comparison}
On $X_{\pe}^\omega$ we have a filtered isomorphism 
$\wh\dR_{X_\pe/k}^\an \simeq \mathbb{B}_{\dR}^+$ of hypersheaves.
Consequently, the $0$-th cohomology sheaf of $\wh\dR_{X_\pe/k}^\an$
is identified with the sheaf $\mathbb{B}_{\dR}^+$
as filtered sheaves on $X_\pe$.
\end{proposition}

Before the proof,
we want to mention that under the equivalence 
$\scrD(X,k)\cong \Sh^{\mathrm{hyp}}(X,k)$ 
and its filtered version (c.f.~Remark \ref{Rmk, equiv}), 
this Proposition implies that the derived de Rham complex 
$\wh \dR_{X_\pe/k}^\an$
is represented by the ordinary sheaf $\mathbb{B}_{\dR}^+$.
Here the induced filtration on $\scrH^0(\wh\dR_{X_\pe/k}^\an)$ is given by
$\scrH^0(\Fil^*\wh\dR_{X_\pe/k}^\an)$.

\begin{proof}
The first sentence follows from~\Cref{hypersheaf properties} (1).

Given a hypersheaf $F$ supported in cohomological degree $0$ on a basis of a site $S$,
it also defines an ordinary sheaf on $S$ (by taking the $0$-th cohomology).
The unfolding of $F$ is a hypersheaf in $\scrD^{\geq 0}$, and its
$0$-th cohomological sheaf is the ordinary sheaf one obtains.

In our situation, we have the basis $X_{\pe}^\omega$ of the site $X_{\pe}$,
and Scholze's $\mathbb{B}_{\dR}^+$ (and its filtrations) are defined
as the ordinary sheaf obtained from $\mathbb{B}_{\dR}^+(\hat{\cO}^+_X)$ 
(and its $\ker(\theta)$-adic filtrations).
Now previous paragraph and the first statement give us the second statement.
\end{proof}

\begin{definition}
Let $X$ be a local complete intersection rigid space over $k$. 
Then the \emph{Hodge-completed analytic derived de Rham complex} of $X_{\pe}$ over $X$,
denoted by $\wh\dR_{X_\pe/X}^\an$,
is defined to be the unfolding of the hypersheaf on $X_{\pe}^\omega$
whose value at $U = \Spa(B,B^+) \in X_{\pe}^\omega$ is $\wh\dR_{B/X}^\an$.

Similarly we define a filtration on $\wh\dR_{X_\pe/X}^\an$ by unfolding
the Hodge filtration on $\wh\dR_{B/X}^\an$.
So $\wh\dR_{X_\pe/X}^\an$ is also derived complete
with respect to the filtration.
\end{definition}

If $X$ is a local complete intersection rigid space over $k$ with embedded codimension $c$.
Then by~\Cref{cohomological bound aff} (3), we see that
$\wh\dR_{X_\pe/X}^\an$ lives in $\Sh^{{\mathrm{hyp}}}(X_{\pe},\scrD^{\geq -c}(k))$.

The Poincar\'{e} Lemma obtained in the previous subsection now immediately
yields the following:
\begin{theorem}
\label{rational Poincare sequence}
Let $X$ be a local complete intersection rigid space over $k$. 
Then there is a functorial strict exact filtration on $\wh\dR_{X_{\pe}/k}^\an$
witnessing the following:
\[
\wh\dR_{X_{\pe}/k}^\an \rra \wh\dR_{X_{\pe}/X}^\an \xrightarrow{\nabla} 
\dR_{X_{\pe}/X}^\an \otimes_{\nu^{-1}\cO_X} \st_1(\nu^{-1}(\LL_{X/k}^\an))
\xrightarrow{\nabla} \cdots
\]

If $X$ is further assumed to be smooth over $k$ of equidimension $d$, 
then the following $\wh\dR_{X_{\pe}/k}^\an$-linear sequence
\begin{align*}
0 \to \wh\dR_{X_{\pe}/k}^\an \rra \wh\dR_{X_{\pe}/X}^\an \xrightarrow{\nabla} 
\dR_{X_{\pe}/X}^\an \otimes_{\nu^{-1}\cO_X} \st_1(\nu^{-1}(\LL_{X/k}^\an)) \xrightarrow{\nabla} \cdots \\
\cdots \xrightarrow{\nabla} \wh\dR_{X_{\pe}/X}^\an \otimes_{\nu^{-1}\cO_X} 
\st_d(\nu^{-1}(\rmL\wedge^d\LL_{X/k}^\an)) \to 0
\end{align*}
is strict exact.
\end{theorem}

Note that as $X/k$ is assumed to be local complete intersection,
these wedge powers of the analytic cotangent complex are (locally) perfect complexes,
hence the completed tensor is the same as just tensor.

\begin{proof}
Since both of unfolding and taking $\Gr^i$ commute with taking limit, the above follows 
from unfolding~\Cref{Poincare Lemma},
and the fact that the completed tensor in~\Cref{Poincare Lemma}
is the same as tensor for local complete intersections $X/k$.

When $X$ is smooth over $k$, everything in sight
(on the basis of affinoid perfectoids in $X_{\pe}^\omega$)
are supported cohomologically in
degree $0$ with filtrations given by submodules because 
of~\Cref{ddR aff computation},~\Cref{cohomological bound aff}, 
and~\Cref{hypersheaf properties},
the strict exact Katz–Oda filtration gives what we want.
\end{proof}

\subsection{Comparing with Scholze's de Rham period sheaf}
In this subsection we show that when $X$ is smooth, 
the de Rham sheaf $\wh\dR_{X_{\pe}/X}^\an$ defined above is related to Scholze's 
de Rham period sheaf $\cO\mathbb{B}_{\dR}^+$. 
We refer readers to~\cite[part (3)]{SchCor} for the its definition.
Following notation of loc.~cit.,
let $\Spa(R_i, R_i^+)$ be an affinoid perfectoid in $X_{\pe}$
with $\Spa(R_0, R_)^+)$ an affinoid open in $X$.
Then for any $i$, we have maps
\[
R_i^+ \ra \dR_{R^+/R_i^+}^\an \text{ and } 
\mathbb{A}_{inf}(R^+) = \dR_{R^+/W(\kappa)}^\an \ra \dR_{R^+/R_i^+}^\an,
\]
which is compatible with maps to $R^+$,
here $\kappa$ denotes the residue field of $k$.
The equality above is deduced from~\Cref{dR when relatively perfect} (1).
Therefore we get an induced map 
\[
R_i^+ \hat{\otimes}_{W(\kappa)} \mathbb{A}_{inf}(R^+) \ra \dR_{R^+/R_i^+}^\an \ra \wh\dR_{R/R_i}^\an.
\]
Taking the composition map above,
inverting $p$ and completing along the kernel of the surjection onto $R$
(note that $\wh\dR_{R/R_i}^\an$ lives in cohomological degree $0$ by~\Cref{cohomological bound aff} (2)
and is already complete with respect to this filtration), 
we get a natural arrow:
\[
\left( (R_i^+ \hat{\otimes}_{W(\kappa)} \mathbb{A}_{inf}(R^+))[1/p] \right)^{\wedge}
\rra  \wh\dR_{R/R_i}^\an \cong \wh\dR_{R/R_0}^\an,
\]
here we apply~\Cref{Poincare Lemma} to $(R_0, R_0^+) \to (R_i, R_i^+) \to (R, R^+)$
to see the filtered isomorphism above.
This arrow is compatible with index $i$, hence after taking colimit,
we get the following map of sheaves on $X_{\pe}^\omega$ 
(see the discussion before~\Cref{hypersheaf properties} for the meaning of $X_{\pe}^\omega$):
\[
f \colon \cO\mathbb{B}_{\dR}^+\mid_{X_{\pe}^\omega} \rra \wh\dR_{X_{\pe}/X}^\an,
\]
which is compatible with maps to $\hat{\cO}_{X_\pe}$
and maps from $\wh\dR_{X_\pe/k}^\an \simeq \mathbb{B}_{\dR}^+$.

\begin{theorem}
\label{OBdR+ comparison}
The map $f$ above induces a filtered isomorphism of sheaves on $X_{\pe}^\omega$.
Hence we get that $\cO\mathbb{B}_{\dR}^+$ is the $0$-th
cohomology sheaf of the hypersheaf $\wh\dR_{X_{\pe}/X}^\an$ on $X_{\pe}$.
\end{theorem}
Similar to~\Cref{BdR+ comparison}, 
under the equivalence $\scrD(X,k)\cong \Sh^{\mathrm{hyp}}(X,k)$ 
and its filtered version (c.f.~Remark \ref{Rmk, equiv}), 
this Theorem implies that
the derived de Rham complex $\wh\dR_{X_{\pe}/X}^\an$ 
is represented by the ordinary sheaf $\cO\mathbb{B}_{\dR}^+$.

\begin{proof}
The second sentence follows from the first sentence, due to the same argument
in the proof of the second statement of~\Cref{BdR+ comparison}.
So it suffices to show the first statement.

On both sheaves, there are natural filtrations:
on $\cO\mathbb{B}_{\dR}^+$ we have the $\ker(\theta)$-adic filtration
where $\theta \colon \cO\mathbb{B}_{\dR}^+ \to \hat{\cO}_{X_\pe}$
and on $\wh\dR_{X_{\pe}/X}^\an$ we have the Hodge filtration with
the first Hodge filtration being kernel of 
$\wh\dR_{X_{\pe}/X}^\an \twoheadrightarrow \hat{\cO}_{X_\pe}$.
Since $f$ is compatible with maps to $\hat{\cO}_{X_\pe}$ and the Hodge filtration is multiplicative,
it suffices to show that $f$ induces an isomorphism on their graded pieces.
Now locally on $X_{\pe}^\omega$, we have that
$\Gr^*(\cO\mathbb{B}_{\dR}^+) \cong \Symm^*_{\hat{\cO}_{X_\pe}}(\Gr^1\cO\mathbb{B}_{\dR}^+)$
by~\cite[Proposition 6.10]{Sch13} and similarly
$\Gr^*(\wh\dR_{X_{\pe}/X}^\an) \cong \Symm^*_{\hat{\cO}_{X_\pe}}(\Gr^1\wh\dR_{X_{\pe}/X}^\an)$
by~\Cref{ddR aff computation} (note that in characteristic $0$ divided powers are the same as symmetric powers).
Therefore we have reduced ourselves to showing that $f$ induces an isomorphism on the first graded pieces.
Their first graded pieces admits a common submodule given by the first graded pieces of
$\wh\dR_{X_\pe/k}^\an \simeq \mathbb{B}_{\dR}^+$ which is $\hat{\cO}_{X_\pe}(1)$.

Now we get the following diagram:
\[
\xymatrix{
\hat{\cO}_{X_\pe}(1) \ar^{\cong}[d] \ar[r] & \Gr^1\cO\mathbb{B}_{\dR}^+ \ar^{\Gr^1 f}[d] \ar[r] & 
\hat{\cO}_{X_\pe} \otimes_{\cO_X} \Omega^\an_X \ar^{g}[d] \\
\hat{\cO}_{X_\pe}(1) \ar[r] & \Gr^1\wh\dR_{X_{\pe}/X}^\an \ar[r] &
\hat{\cO}_{X_\pe} \otimes_{\cO_X} \Omega^\an_X 
}
\]
with both rows being short exact 
(by~\cite[Corollary 6.14]{Sch13} and~\Cref{ddR aff computation} respectively)
and the left vertical arrow being an isomorphism as $f$ is compatible with the maps from
$\wh\dR_{X_\pe/k}^\an \simeq \mathbb{B}_{\dR}^+$,
which is why we get the induced arrow $g$.
Moreover $f$ is linear over $\wh\dR_{X_\pe/k}^\an \simeq \mathbb{B}_{\dR}^+$,
which implies that $g$ is linear over $\hat{\cO}_{X_\pe}$.
Therefore it suffices to show that $g$ induces an isomorphism.

As the statement is \'etale local, we may assume that
$X = \mathbb{T}^n = 
\Spa(k\langle T_1^{\pm1}, \ldots, T_n^{\pm1} \rangle, \cO_k\langle T_1^{\pm1}, \ldots, T_n^{\pm1} \rangle)$.
Denote $\mathbb{T}^n_{\infty}$
the pro-finite-\'{e}tale tower above $\mathbb{T}^n$
given by adjoining $p$-power roots of the coordinates $T_i$.
We have the following diagram
\[
\xymatrix{
\mathbb{Z}_p \langle T_i^{\pm 1}, S_i^{1/p^\infty} \rangle = 
\mathbb{Z}_p \langle T_i^{\pm 1} \rangle \hat{\otimes}_{\mathbb{Z}_p} 
\mathbb{Z}_p \langle S_i^{1/p^\infty} \rangle \ar^-{\alpha}[r] \ar^-{\beta}[d] &
\cO\BB_{\dR}^{+}\mid_{\mathbb{T}^n_{\infty}} \ar^-{f}[d] \\
\mathbb{Q}_p \langle S_i^{\pm 1/p^\infty} \rangle [\![X_i]\!] \ar^-{\gamma}[r] &
\wh\dR^\an_{X_\pe/X}\mid_{\mathbb{T}^n_{\infty}}.
}
\]
Here the arrow $\beta$ is given by sending $T_i$ to $X_i + S_i$,
and $S_i$ is sent to $1 \otimes [T_i^\flat]$ under $\alpha$.
The element $\alpha(T_i - S_i)$ is $u_i \in \Fil^1\cO\BB_{\dR}^{+}$
whose image in $\hat{\cO}_{X_\pe} \otimes_{\cO_X} \Omega^\an_X $
is $1 \otimes dT_i$, see the discussion before~\cite[Proposition 6.10]{Sch13}.
On the other hand, the element $\beta(T_i - S_i)$ is $X_i$,
and the image of $\gamma(X_i)$ in $\hat{\cO}_{X_\pe} \otimes_{\cO_X} \Omega^\an_X $
is also $1 \otimes dT_i$ by~\Cref{rational torus example},~\Cref{cotangent complex of perfectoid circle}
and~\Cref{dR of perfectoid circle}.
Therefore we get that $g(1 \otimes dT_i) = 1 \otimes dT_i$, since $g$ is linear over $\hat{\cO}_{X_\pe}$
and $\Omega^\an_X$ is generated by $dT_i$'s, we see that $g$ is an isomorphism, hence finishes the proof.
\end{proof}

\begin{remark}
\label{rational Poincare matches}
In the process of the proof above, we also see that under the identification
in~\Cref{BdR+ comparison} and~\Cref{OBdR+ comparison}, the Poincar\'{e} sequence
obtained in~\Cref{rational Poincare sequence} and the one in Scholze's paper~\cite[Corollary 6.13]{Sch13}
matches, c.f.~proof of the second statement of~\Cref{crystalline period comparison}.

Also the Faltings' extension (see~\cite[Corollary 6.14]{Sch13} and~\Cref{ddR aff computation}), 
being the first graded pieces of $\cO\mathbb{B}_{\dR}^+ \cong \scrH^0(\wh\dR_{X_{\pe}/X}^\an)$,
is matched up.
In some sense, our proof above reduces to identifying the Faltings' extension,
and this is a well-known fact to experts.
In fact, this project was initiated after Bhargav Bhatt explained to us how to get Faltings' extension
from the analytic cotangent complex $\LL^\an_{X_\pe/X}$.
\end{remark}

\subsection{An example}
\label{artinian example}
In this complementary subsection, we would like to compute the Hodge-completed
analytic derived de Rham complex of a perfectoid algebra over a $0$-dimensional
$k$-affinoid algebra.
Surprisingly, the underlying algebra (forgetting its filtration) one get
always lives in cohomological degree $0$, which leads us to the~\Cref{coh bound question}.

Without loss of generality, let $(K, K^+)$ be a perfectoid field
over $k$, containing all $p$-power
roots of unity, and let $A$ be an Artinian local finite $k$-algebra
with residue field being $k$ as well.
Let $(B,B^+)$ be a perfectoid affinoid algebra containing $(K,K^+)$ and
let $A \to B$ be a morphism of $k$-algebras.
Since perfectoid affinoid algebras are reduced, we get a sequence of maps
$k \to A \to k \to B$.

By the above sequence, we get natural filtered $k$-linear maps:
\[
\wh\dR^\an_{B/k} \rra \wh\dR^\an_{B/A} \rra \wh\dR^\an_{B/k} \text{ and } \wh\dR^\an_{k/A} \rra \wh\dR^\an_{B/A}.
\]
This induces a filtered map:
\[
\wh\dR^\an_{B/k} \otimes_k \wh\dR^\an_{k/A} \rra \wh\dR^\an_{B/A}
\]
where the filtration on the source comes from the symmetric monoidal structure on $\DF(k)$.
Since this map is compatible with the filtration and the target is complete with respect
to its filtration, we get an induced map:
\[
\wh\dR^\an_{B/k} \hat{\otimes}_k \wh\dR^\an_{k/A} \rra \wh\dR^\an_{B/A}.
\]

\begin{proposition}
The map $\wh\dR^\an_{B/k} \hat{\otimes}_k \wh\dR^\an_{k/A} \rra \wh\dR^\an_{B/A}$
above is a filtered isomorphism.
\end{proposition}

\begin{proof}
Since both are complete with respect to their filtrations, it suffices to show
the map induces an isomorphism on the graded pieces.
The graded algebra of both sides are the symmetric algebra (over $B$) on their
first graded pieces, hence it suffices to check 
$\Gr^1(\wh\dR^\an_{B/k} \hat{\otimes}_k \wh\dR^\an_{k/A}) \rra 
\Gr^1(\wh\dR^\an_{B/A})$ being an isomorphism.
This follows from the decomposition of analytic cotangent complexes
\[
\LL_{B/A}^\an \cong \LL_{B/k}^\an \oplus (\LL_{A/k}^\an \otimes_A B)
\]
which is deduced from contemplating the sequence
$k \to A \to k \to B$.
\end{proof}

We know that $\wh\dR^\an_{B/k} \cong \mathbb{B}_{\dR}^+(B)$,
a result of Bhatt tells us the underlying algebra of $\wh\dR^\an_{k/A} \cong A$,
explained in below.
Since $A \to k$ is a surjection, the analytic cotangent complex agrees with the classical cotangent complex,
hence we have a filtered isomorphism
\[
\wh\dR^\an_{k/A} \rra \wh\dR_{k/A}.
\]
Now~\cite[Theorem 4.10]{Bha12complete} implies the underlying algebra
$\wh\dR_{k/A}$ is isomorphic to the completion of $A$ along the surjection $A \to k$.
Since $A$ is an Artinian local ring, this completion is simply $A$ itself.
Therefore we get a map of the underlying algebras:
\[
\mathbb{B}_{\dR}^+(B) \otimes_k A \rra \wh\dR^\an_{B/k} \hat{\otimes}_k \wh\dR^\an_{k/A}
\]

\begin{proposition}
The map $\mathbb{B}_{\dR}^+(B) \otimes_k A \rra \wh\dR^\an_{B/k} \hat{\otimes}_k \wh\dR^\an_{k/A}$
above is an isomorphism.
Consequently we have an isomorphism
\[
\mathbb{B}_{\dR}^+(B) \otimes_k A \cong \wh\dR^\an_{B/A}.
\]
\end{proposition}

\begin{proof}
By definition, we have
\[
\wh\dR^\an_{B/k} \hat{\otimes}_k \wh\dR^\an_{k/A} \cong 
\lim_n \lim_m \mathbb{B}_{\dR}^+(B)/(\xi)^n \otimes_k \dR_{k/A}/\Fil^m,
\]
here we have used the (filtered) identification $\wh\dR^\an_{k/A} \cong \wh\dR_{k/A}$
spelled out before this Proposition.

We claim that for any given $n$, we have an isomorphism
\[
\mathbb{B}_{\dR}^+(B)/(\xi)^n \otimes_k A \cong 
\lim_m \mathbb{B}_{\dR}^+(B)/(\xi)^n \otimes_k \dR_{k/A}/\Fil^m.
\]
Indeed for each $i \in \mathbb{Z}$, we have the following short exact sequence:
\[
\xymatrix{
0 \ar[r] & 
\mathrm{R}^1\lim_m \left(\mathbb{B}_{\dR}^+(B)/(\xi)^n \otimes_k \rmH^{i-1}(\dR_{k/A}/\Fil^m)\right) \ar[ld] & \\
 \rmH^i(\lim_m \left(\mathbb{B}_{\dR}^+(B)/(\xi)^n \otimes_k \dR_{k/A}/\Fil^m\right)) \ar[r] &
\lim_m \left(\mathbb{B}_{\dR}^+(B)/(\xi)^n \otimes_k \rmH^{i}(\dR_{k/A}/\Fil^m)\right) \ar[r] & 0 &
}
\]
Since for each $m$ and $i$, the vector space $\rmH^{i-1}(\dR_{k/A}/\Fil^m)$
is finite dimensional over $k$, we see that the inverse system 
$\mathbb{B}_{\dR}^+(B)/(\xi)^n \otimes_k \rmH^{i-1}(\dR_{k/A}/\Fil^m)$
satisfies Mittag-Leffler condition, hence the $\mathrm{R}^1 \lim$ term vanishes.
By~\cite[Theorem 4.10]{Bha12complete}, we have that the inverse system
$\{\rmH^{i}(\dR_{k/A}/\Fil^m)\}_m$ is pro-isomorphic to $0$ if $i \not= 0$
and is pro-isomorphic to $A$ (since $A$ is finite dimensional over $k$) if $i = 0$,
therefore the above short exact sequence becomes
\[
\rmH^i(\lim_m \left(\mathbb{B}_{\dR}^+(B)/(\xi)^n \otimes_k \dR_{k/A}/\Fil^m\right)) \cong
\begin{cases}
0;~i \not= 0\\
\mathbb{B}_{\dR}^+(B)/(\xi)^n \otimes_k A;~i = 0.
\end{cases}
\]

This gives us the claim above.

Now we have
\[
\wh\dR^\an_{B/k} \hat{\otimes}_k \wh\dR^\an_{k/A} \cong 
\lim_n (\lim_m \mathbb{B}_{\dR}^+(B)/(\xi)^n \otimes_k \dR_{k/A}/\Fil^m)
\cong \lim_n (\mathbb{B}_{\dR}^+(B)/(\xi)^n \otimes_k A)
\cong \mathbb{B}_{\dR}^+(B) \otimes_k A
\]
as desired, where the last identification follows from the fact that $A$ is finite over $k$.
\end{proof}

If one contemplates the example $A = k[\epsilon]/(\epsilon^2)$,
one sees that $\dR^\an_{B/A}/\Fil^i$ does not live in cohomological degree $0$
alone for any $i \geq 2$.

As a consequence of the above Proposition, 
for the $X = \mathrm{Spa}(A)$ we have an equality of presheaves
on $X_{\pe}^\omega$:
\[
\wh\dR^\an_{X_\pe/X} \cong \mathbb{B}_{\dR}^+ \otimes_k \nu^{-1}\cO_X,
\]
in particular the underlying algebra of $\wh\dR^\an_{X_\pe/X}$ pro-\'{e}tale locally
lives in cohomological degree $0$.
Motivated by this computation and results in~\cite{Bha12complete}, we end this article
by asking the following:
\begin{question}
\label{coh bound question}
In what generality shall we expect $\wh\dR^\an_{X_\pe/X}\mid_{X_{\pe}^\omega}$ 
to live in cohomological degree $0$?
And when that happens, can we re-interpret the underlying algebra
via some construction similar to Scholze's $\cO\mathbb{B}_{\dR}^+$
as in~\cite{Sch13} and~\cite{SchCor}?
\end{question}

\section{Appendix: local complete intersections in rigid geometry}
\label{lci discussion}
In this appendix we make a primitive discussion of 
local complete intersection morphisms in rigid geometry.
We remark that the results recorded here hold verbatim with $k$ being a general
complete non-Archimedean field.

In order to talk about local complete intersections,
we need to understand how being of finite Tor 
dimension\footnote{In classical literature such as~\cite{Avr99} 
this corresponds to the notion of having finite flat dimension.} 
behaves under base change
in rigid geometry.

\begin{lemma}
\label{surjection finite Tor}
Let $A$ and $ B$ be two affinoid $k$-algebras, and $A\ra B$ a morphism of Tor dimension $m$.
Let $P \coloneqq A \langle T_1, \ldots, T_n \rangle \twoheadrightarrow B$ be a surjection,
then we have 
\[
\mathrm{Tor}\dim_P(B) \leq m + n.
\]
\end{lemma}

The following proof is suggested to us by Johan de Jong.

\begin{proof}
Choose a resolution of $B$ by finite free $P$-modules
\[
\ldots \xrightarrow{d_i} M_i \xrightarrow{d_{i-1}} M_{i-1} \ldots \xrightarrow{d_0} M_0 \twoheadrightarrow B.
\]
Since $P$ is flat over $A$, we see that $M \coloneqq \mathrm{Coker}(d_m)$ is flat over $A$ as
$A \to B$ is assumed to be of Tor dimension 
$m$~\cite[\href{https://stacks.math.columbia.edu/tag/0653}{Tag 0653}]{Sta}.
Moreover $M$ is finitely generated over $P$ since $P$ is Noetherian.
Now we use~\cite[Lemma 6.3]{Li19} to see that $M$ admits a projective resolution
over $P$ of length $n$.
Therefore we get that $B$ has a projective resolution over $P$ of length $m + n$.
\end{proof}

\begin{lemma}
\label{base change finite Tor}
Let $A$ and $ B$ be two affinoid $k$-algebras, and $A\ra B$ a morphism of finite Tor dimension.
Let $C$ be any affinoid $A$-algebra, then the base change (in the realm of rigid geometry)
$C \to B \hat{\otimes}_A C$ is also of finite Tor dimension.
\end{lemma}

\begin{proof}
Choose a surjection $A \langle T_1, \ldots, T_n \rangle \twoheadrightarrow B$,
which again is of finite Tor dimension by~\Cref{surjection finite Tor}.
Then we have a factorization:
\[
C \to C\langle T_1, \ldots, T_n \rangle \to 
B \otimes_{A \langle T_1, \ldots, T_n \rangle} C \langle T_1, \ldots, T_n \rangle \cong B \hat{\otimes}_A C.
\]
Since the first arrow is flat and the second arrow,
being base change of an arrow of finite Tor dimension, is of finite Tor dimension,
we conclude that the composition is of finite Tor 
dimension~\cite[\href{https://stacks.math.columbia.edu/tag/066J}{Tag 066J}]{Sta}.
\end{proof}

\begin{proposition}
\label{lci properties}
Let $A \to B$ a morphism of $k$-affinoid algebras.
Then the following are equivalent:
\begin{enumerate}
\item any surjection $A\langle T_1, \ldots, T_n\rangle \twoheadrightarrow B$ is a local complete intersection;
\item there exists a surjection $A\langle T_1, \ldots, T_n\rangle \twoheadrightarrow B$ which 
is a local complete intersection;
\item $A \to B$ is of finite Tor dimension 
and the analytic cotangent complex $\LL_{B/A}^\an$ is a perfect $B$-complex.
\end{enumerate}
Moreover, any of these three equivalent conditions implies that $\LL_{B/A}^\an$
is a perfect complex with Tor amplitude in $[-1,0]$.
\end{proposition}

\begin{proof}
It is easy to see that (1) implies (2).

To see (2) implies (3), first of all
$A\langle T_1, \ldots, T_n\rangle \twoheadrightarrow B$ is a local complete intersection
implies that it is of finite Tor dimension.
Since $A \to A\langle T_1, \ldots, T_n\rangle$ is flat,
we see that $A \to B$ is also finite Tor dimension 
by~\cite[\href{https://stacks.math.columbia.edu/tag/0653}{Tag 0653}]{Sta}.
Next we look at the triangle $A \to A\langle T_1, \ldots, T_n\rangle \to B$,
which gives rise to a triangle of analytic cotangent complexes:
\[
\LL^\an_{A\langle T_1, \ldots, T_n\rangle/A} \otimes_A B \to \LL^\an_{B/A} \to 
\LL^\an_{B/A\langle T_1, \ldots, T_n\rangle}.
\]
Now~\Cref{GR statements} (3) gives that the first term is a perfect complex with Tor amplitude in $[0,0]$,
while condition (2) and~\Cref{GR statements}.(4) implies that the third term is 
a perfect complex with Tor amplitude in $[-1,-1]$,
hence we see that (2) implies (3) and gives the last sentence as well.

Lastly we need to show that (3) implies (1).
To that end we apply Avramov's solution of Quillen's conjecture~\cite{Avr99}.
As $A \to B$ is of finite Tor dimension,
we see that any surjection $A\langle T_1, \ldots, T_n\rangle \twoheadrightarrow B$
has finite Tor dimension by~\Cref{surjection finite Tor}.
The previous paragraph shows that $\LL_{B/A}^\an$ being a perfect complex
is equivalent to the classical cotangent complex 
$\LL_{B/A\langle T_1, \ldots, T_n\rangle}$ being a perfect complex.
Now we use Avramov's result~\cite[Theorem 1.4]{Avr99} to conclude that 
$A\langle T_1, \ldots, T_n\rangle \twoheadrightarrow B$ is a local complete intersection.
\end{proof}

\begin{definition}
\label{lci def}
Let $A \to B$ be a morphism of $k$-affinoid algebras.
The morphism $A \to B$ of $k$-affinoid algebras
is called \emph{a local complete intersection} if one of the three equivalent conditions
in~\Cref{lci properties} is satisfied.

Let $Y \to X$ be a morphism of rigid spaces over $k$.
Then this morphism is called \emph{a local complete  intersection}
if for any pair of affinoid domains $U$ and $V$ in $X$ and $Y$,
such that the image of $V$ is contained in $U$,
the induced map of $k$-affinoid algebras
is a local complete intersection.
\end{definition}

We leave it as an exercise (using~\Cref{GR statements}) that a morphism being
a local complete intersection may be checked locally on the source and target.
We caution readers that there is a notion of local complete intersection morphism
between Noetherian rings, 
while the notion we define here should (clearly) only be considered in the situation
of rigid geometry. 
These two notions agree when the morphism considered is a surjection.
We hope this slight abuse of notion will not cause any confusion.
But as a sanity check, let us show here that this notion matches
the corresponding notion in classical algebraic geometry  
under rigid-analytification.
The following is suggested to us by David Hansen.

\begin{proposition}
Let $f \colon X \to Y$ be a morphism of schemes locally of finite type over 
a $k$-affinoid algebra $A$
with rigid-analytification $f^\an \colon X^\an \to Y^\an$.
Then $f$ is a local complete intersection (in the classical sense)
if and only if $f^\an$ is a local complete intersection (in the sense of~\Cref{lci def}).
\end{proposition}

\begin{proof}
We first reduce to the case where both of $X$ and $Y$ are affine.
Then we may check this after fiber product $Y$ with an affine space so that
$f$ is a closed embedding.
In this situation, we have identification of ringed sites $X^\an \cong X \times_Y Y^\an$
and an identification of cotangent complexes:
\[
\iota^*\LL_{X/Y} \simeq \LL^\an_{X^\an/Y^\an},
\]
where $\iota \colon X^\an \to X$ is the natural map of ringed sites.

Now we use the fact that classical Tate points on $X^\an$ is in bijection
with closed points on $X$, and for any such point $x$, the map
$\iota^\sharp \colon \cO_{X,x} \to \cO_{X^\an,x}$ of local rings is faithfully flat.
Therefore we can check $\LL_{X/Y}$ being perfect by pulling back along
$\iota$, hence $\LL_{X/Y}$ is perfect if and only if $\LL^\an_{X^\an/Y^\an}$
is perfect, and this finishes the proof.
\end{proof}

%

Next we turn to understand the localization of analytic cotangent complexes
for a local complete intersection morphism.

Let us introduce some notions:
\begin{definition}
Let $A \to B$ be a morphism of $k$-affinoid algebras.
Let $\mathfrak{m} \subset B$ be a maximal ideal, the \emph{embedded dimension of $B/A$
at $\mathfrak{m}$} is defined to be the following
\[
\dim_{B/A,\mathfrak{m}} \coloneqq \dim_{\kappa(\mathfrak{m})} (\Omega^\an_{B/A} \otimes_B B/\mathfrak{m}).
\]

Let $\mathfrak{n}$ be the preimage of $\mathfrak{m}$ in $A$ (which is also a maximal ideal),
we define the \emph{embedded codimension of $B/A$ at $\mathfrak{m}$} to be
\[
\dim_{B/A,\mathfrak{m}} + \dim(A_\mathfrak{n}) - \dim(B_\mathfrak{m}).
\]

The \emph{embedded codimension of $B/A$} is the supremum of that at all maximal ideals 
$\mathfrak{m} \subset B$.
\end{definition}

\begin{proposition}
\label{embedded codimension control}
Let $A \to B$ be a local complete intersection morphism of $k$-affinoid algebras.
Then at any maximal ideal $\mathfrak{m} \subset B$, there is a presentation of the analytic cotangent complex
\[
\LL^\an_{B/A} \otimes_B B_{\mathfrak{m}} \simeq 
\left[B_{\mathfrak{m}}^{\oplus c(\mathfrak{m})} \to B_{\mathfrak{m}}^{\oplus d(\mathfrak{m})}\right]
\]
where $c(\mathfrak{m})$ is the embedded codimension of $B/A$ at $\mathfrak{m}$
and $d(\mathfrak{m})$ is the embedded dimension of $B/A$ at $\mathfrak{m}$.
Here $B_{\mathfrak{m}}^{\oplus d(\mathfrak{m})}$ is put in degree $0$.

In particular the Tor amplitude of $\rmL \Symm^i \LL^\an_{B/A}$ is always in $[-\min\{c,i\},0]$
where $c$ is the embedded codimension of $B/A$.
\end{proposition}

\begin{proof}
We may always replace $B$ by a rational domain containing the point $\mathfrak{m}$ 
(viewed as a classical Tate point on the associated adic space), so we can assume there are
power bounded elements $f_1, \ldots, f_{d(\mathfrak{m})}$ whose differentials
generate the stalk of $\Omega^\an_{B/A}$
at $\mathfrak{m}$.
Thus we have a map $A' \coloneqq A\langle T_1, \ldots, T_{d(\mathfrak{m})} \rangle \to B$
which is unramified at $\mathfrak{m}$, see~\cite[Section 1.6]{Huber}.
By Proposition 1.6.8 of loc.~cit.~we can factortize the map $A' \to B$ as
$A' \xrightarrow{h} C \xrightarrow{g} B$ where $h$ is \'{e}tale and $g$ is surjective.

One checks that the \'{e}taleness of $h$ guarantees that the surjection $C \xrightarrow{g} B$
has finite Tor dimension.
Moreover~\Cref{GR statements} implies that $\LL_{B/C}$ is a perfect complex because of the triangle
\[
\LL^\an_{C/A} \otimes_C B \ra \LL^\an_{B/A} \ra \LL_{B/C}.
\]
Hence $C \to B$ is a surjective local complete intersection.
Hence the kernel of $C \to B$
around $\mathfrak{m}$ is generated by a length $c(\mathfrak{m})$ regular sequence.
This in turn implies that $\LL_{B/C} \otimes_B B_{\mathfrak{m}} \simeq 
B_{\mathfrak{m}}^{\oplus c(\mathfrak{m})}[1]$,
which together with the triangle above gives the local presentation we want in the statement.

The statement concerning Tor amplitude can be checked at every maximal ideal which,
by our presentation, follows from the formula $\rmL \Symm^i(C[1]) \simeq \rmL \wedge^i(C)[i]$,
see~\cite[V.4.3.4]{Ill71}.
\end{proof}

\bibliographystyle{amsalpha}
\bibliography{template}

\providecommand{\bysame}{\leavevmode\hbox to3em{\hrulefill}\thinspace}
\providecommand{\MR}{\relax\ifhmode\unskip\space\fi MR }
\providecommand{\MRhref}[2]{%
  \href{http://www.ams.org/mathscinet-getitem?mr=#1}{#2}
}
\providecommand{\href}[2]{#2}
\begin{thebibliography}{{Sta}20}

\bibitem[AI13]{AI13}
Fabrizio Andreatta and Adrian Iovita, \emph{Comparison isomorphisms for smooth
  formal schemes}, J. Inst. Math. Jussieu \textbf{12} (2013), no.~1, 77--151.
  \MR{3001736}

\bibitem[Ant20]{Ant20}
Jorge Ant\'onio, \emph{Spreading out the {H}odge filtration in the
  non-archimedean geometry}, 2020, arXiv:2005.00774, available at
  \url{https://arxiv.org/abs/2005.00774}.

\bibitem[Avr99]{Avr99}
Luchezar~L. Avramov, \emph{Locally complete intersection homomorphisms and a
  conjecture of {Q}uillen on the vanishing of cotangent homology}, Ann. of
  Math. (2) \textbf{150} (1999), no.~2, 455--487. \MR{1726700}

\bibitem[Bei12]{Bei12}
A.~Beilinson, \emph{{$p$}-adic periods and derived de {R}ham cohomology}, J.
  Amer. Math. Soc. \textbf{25} (2012), no.~3, 715--738. \MR{2904571}

\bibitem[Ber74]{Ber74}
Pierre Berthelot, \emph{Cohomologie cristalline des sch\'{e}mas de
  caract\'{e}ristique {$p>0$}}, Lecture Notes in Mathematics, Vol. 407,
  Springer-Verlag, Berlin-New York, 1974. \MR{0384804}

\bibitem[Bha12a]{Bha12complete}
Bhargav Bhatt, \emph{Completions and derived de rham cohomology}, 2012,
  arXiv:1207.6193, available at \url{https://arxiv.org/abs/1207.6193}.

\bibitem[Bha12b]{Bha12}
\bysame, \emph{p-adic derived de rham cohomology}, 2012, arXiv:1204.6560,
  available at \url{https://arxiv.org/abs/1204.6560}.

\bibitem[BMS19]{BMS2}
Bhargav Bhatt, Matthew Morrow, and Peter Scholze, \emph{Topological
  {H}ochschild homology and integral {$p$}-adic {H}odge theory}, Publ. Math.
  Inst. Hautes \'{E}tudes Sci. \textbf{129} (2019), 199--310. \MR{3949030}

\bibitem[Bri08]{Brinon}
Olivier Brinon, \emph{Repr\'{e}sentations {$p$}-adiques cristallines et de de
  {R}ham dans le cas relatif}, M\'{e}m. Soc. Math. Fr. (N.S.) (2008), no.~112,
  vi+159. \MR{2484979}

\bibitem[Col12]{Col12}
Pierre Colmez, \emph{Une construction de {$\bold B_{\rm dR}^+$}}, Rend. Semin.
  Mat. Univ. Padova \textbf{128} (2012), 109--130 (2013). \MR{3076833}

\bibitem[Fal89]{Fal89}
Gerd Faltings, \emph{Crystalline cohomology and {$p$}-adic
  {G}alois-representations}, Algebraic analysis, geometry, and number theory
  ({B}altimore, {MD}, 1988), Johns Hopkins Univ. Press, Baltimore, MD, 1989,
  pp.~25--80. \MR{1463696}

\bibitem[Fon94]{Fon94}
Jean-Marc Fontaine, \emph{Le corps des p\'{e}riodes {$p$}-adiques}, no. 223,
  1994, With an appendix by Pierre Colmez, P\'{e}riodes $p$-adiques
  (Bures-sur-Yvette, 1988), pp.~59--111. \MR{1293971}

\bibitem[FvdP04]{FvdP}
Jean Fresnel and Marius van~der Put, \emph{Rigid analytic geometry and its
  applications}, Progress in Mathematics, vol. 218, Birkh\"{a}user Boston,
  Inc., Boston, MA, 2004. \MR{2014891}

\bibitem[GR03]{GR03}
Ofer Gabber and Lorenzo Ramero, \emph{Almost ring theory}, Lecture Notes in
  Mathematics, vol. 1800, Springer-Verlag, Berlin, 2003. \MR{2004652}

\bibitem[Guo]{Guo}
Haoyang Guo, \emph{Crystalline cohomology of rigid analytic spaces}, in
  preparation.

\bibitem[Hub96]{Huber}
Roland Huber, \emph{\'{E}tale cohomology of rigid analytic varieties and adic
  spaces}, Aspects of Mathematics, E30, Friedr. Vieweg \& Sohn, Braunschweig,
  1996. \MR{1734903}

\bibitem[Ill71]{Ill71}
Luc Illusie, \emph{Complexe cotangent et d\'{e}formations. {I}}, Lecture Notes
  in Mathematics, Vol. 239, Springer-Verlag, Berlin-New York, 1971.
  \MR{0491680}

\bibitem[Ill72]{Ill72}
\bysame, \emph{Complexe cotangent et d\'{e}formations. {II}}, Lecture Notes in
  Mathematics, Vol. 283, Springer-Verlag, Berlin-New York, 1972. \MR{0491681}

\bibitem[KO68]{KO68}
Nicholas~M. Katz and Tadao Oda, \emph{On the differentiation of de {R}ham
  cohomology classes with respect to parameters}, J. Math. Kyoto Univ.
  \textbf{8} (1968), 199--213. \MR{237510}

\bibitem[Li19]{Li19}
Shizhang Li, \emph{On rigid varieties with projective reduction}, J. Algebraic
  Geom. (2019), published electronically at
  \url{https://doi.org/10.1090/jag/740}.

\bibitem[Lur17]{Lu17}
Jacob Lurie, \emph{Higher algebra}.

\bibitem[Lur18]{Lu18}
\bysame, \emph{Spectral algebraic geometry}.

\bibitem[Sch12]{Sch12}
Peter Scholze, \emph{Perfectoid spaces}, Publ. Math. Inst. Hautes \'{E}tudes
  Sci. \textbf{116} (2012), 245--313. \MR{3090258}

\bibitem[Sch13]{Sch13}
\bysame, \emph{{$p$}-adic {H}odge theory for rigid-analytic varieties}, Forum
  Math. Pi \textbf{1} (2013), e1, 77. \MR{3090230}

\bibitem[Sch16]{SchCor}
\bysame, \emph{{$p$}-adic {H}odge theory for rigid-analytic
  varieties---corrigendum {[{MR}3090230]}}, Forum Math. Pi \textbf{4} (2016),
  e6, 4. \MR{3535697}

\bibitem[{Sta}20]{Sta}
The {Stacks Project Authors}, \emph{\textit{Stacks Project}},
  \url{https://stacks.math.columbia.edu}, 2020.

\bibitem[TT19]{TT19}
Fucheng Tan and Jilong Tong, \emph{Crystalline comparison isomorphisms in
  {$p$}-adic {H}odge theory: the absolutely unramified case}, Algebra Number
  Theory \textbf{13} (2019), no.~7, 1509--1581. \MR{4009670}

\end{thebibliography}

\end{document}